\documentclass[a4paper,10pt]{article}
\usepackage[utf8x]{inputenc}
\usepackage{tracefnt,amsmath,tabu,array}
\usepackage{amssymb,graphicx,setspace,amsfonts,amsbsy}
\usepackage{pifont,latexsym,ifthen,amsthm,rotating,calc,textcase,booktabs}

\newtheorem{theorem}{Theorem}[section]
\newtheorem{lemma}[theorem]{Lemma}
\newtheorem{corollary}[theorem]{Corollary}
\newtheorem{proposition}[theorem]{Proposition}
\newtheorem{definition}[theorem]{Definition}
\newtheorem{statement}[theorem]{Statement}

\newtheorem{remark}[theorem]{Remark}
\newcommand{\filledbox}{\leavevmode
  \hbox to.77778em{%
  \hfil\vbox to.675em{\hrule width.6em height.6em}\hfil}}

\DeclareMathOperator*{\esssup}{ess\,sup}

\newcommand{\Rm}{{\mathbb R}}

\newcommand{\Di}{{\mathbf T}}
\newcommand{\eps}{\varepsilon}

\title{Bounded Solutions to an Energy Subcritical Non-linear Wave Equation on $\Rm^3$}
\author{Ruipeng Shen}

\begin{document}
\maketitle

\begin{abstract}
In this work we consider an energy subcritical semi-linear wave equation ($3 < p < 5$)
\[
 \left\{\begin{array}{ll}
  \partial_t^2 u - \Delta u = \phi(x) |u|^{p-1} u,  & (x,t) \in \Rm^3 \times \Rm;\\
  u|_{t=0} = u_0 \in \dot{H}^{s_p}  (\Rm^3); & \\
  \partial_t u|_{t =0} = u_1 \in \dot{H}^{s_p-1} (\Rm^3); &
 \end{array}\right.
\]
where $s_p = 3/2 - 2/(p-1)$ and the function $\phi: \Rm^3 \rightarrow [-1,1]$ is a radial continuous function with a limit at infinity.
We prove that unless the elliptic equation $-\Delta W = \phi(x) |W|^{p-1} W$ has a nonzero radial solution $W \in C^2 (\Rm^3) \cap \dot{H}^{s_p} (\Rm^3)$, any radial solution $u$ with a finite uniform upper bound on the critical Sobolev norm $\|(u(\cdot,t), \partial_t u(\cdot,t))\|_{\dot{H}^{s_p}\times \dot{H}^{s_p}(\Rm^3)}$ for all $t$ in the maximal lifespan must be a global solution in time and scatter.
\end{abstract}

\section{Introduction}
\paragraph{Pure Power-type Nonlinearity} The nonlinear wave equation ($s_p = \frac{3}{2} - \frac{2}{p-1}$)
\[
\left\{\begin{array}{ll}
  \partial_t^2 u - \Delta u = \zeta |u|^{p-1} u,  & (x,t) \in \Rm^3 \times \Rm;\\
  u|_{t=0} = u_0 \in \dot{H}^{s_p}  (\Rm^3); & \\
  \partial_t u|_{t =0} = u_1 \in \dot{H}^{s_p-1} (\Rm^3); &
 \end{array}\right. \qquad (CP0)
\]
has been extensively studied in a lot of previous works. There are two different cases: the defocusing one with $\zeta = -1$ and the focusing one with $\zeta =1$. The latter case is usually more complicated and difficult to deal with. If the initial data are small, the sign of $\zeta$ does not play an important rule. For example, if $p> 1 + \sqrt{2}$, global existence and well-posedness of solutions with small initial data was proved in the papers \cite{smallgs2, smallgs3, gwpwrn, smallgs1} with even worse nonlinear term $|u|^p$. However, the behaviour of solutions is much different in these two cases if the initial data are large. For instance, if the nonlinear term is energy critical ($p=5$), any solution with a finite energy always exists for all time $t$ and scatters in the defocusing case, see \cite{mg1, mg2, ss1, ss2}. On the other hand, the solutions to the energy critical, focusing equation may scatter, blow up in finite time or stay unchanged for all time, as shown in \cite{secret, tkm1, kenig}. There are also lots of works on the energy subcritical case ($p < 5$, see \cite{kv1, shen2}) or the energy supercritical case ($p > 5$, see \cite{dkm2, km, kv2, kv3}).

\paragraph{Topic of this work} We consider a semi-linear wave equation with an more general energy subcritical nonlinearity ($3 < p < 5$)
\[
 \left\{\begin{array}{ll}
  \partial_t^2 u - \Delta u = \phi(x) |u|^{p-1} u,  & (x,t) \in \Rm^3 \times \Rm;\\
  u|_{t=0} = u_0 \in \dot{H}^{s_p}  (\Rm^3); & \\
  \partial_t u|_{t =0} = u_1 \in \dot{H}^{s_p-1} (\Rm^3); &
 \end{array}\right. \qquad (CP1)
\]
with radial data. Here $s_p = 3/2 - 2/(p-1)$ and the function $\phi: \Rm^3 \rightarrow [-1,1]$ is a radial continuous function with a well-defined limit
\[
 \lim_{|x| \rightarrow \infty} \phi (x) = \phi(\infty).
\]
Let us recall the main conclusion\footnote{This is slightly different from the original result, as its uniform boundedness condition is only concerning the positive time direction. But a careful review on the original proof reveals that this different version of theorem still holds.} of my previous work \cite{shen2} on the special case $\phi(x) \equiv \pm 1$
\begin{theorem} \label{theorem with phi pm 1}
Assume $3 < p < 5$. Let $u$ be a radial solution to the equation (CP0) with a maximal lifespan $(-T_-, T_+)$ and a uniform boundedness condition
\[
 \sup_{t \in [0, T_+)} \left\|(u(\cdot, t), \partial_t u(\cdot, t))\right\|_{\dot{H}^{s_p} \times \dot{H}^{s_p-1}(\Rm^3)} < \infty.
\]
Then $T_+ = + \infty$ and $u$ scatters in the positive time direction.
\end{theorem}
The meaning of ``scattering'' here is explained in Remark \ref{three types of solutions}. In this work we generalize this result and prove
\begin{theorem} \label{main theorem work}
Any radial solution $u$ to (CP1) with a maximal lifespan $(-T_-,T_+)$ satisfying
\[
 \sup_{t \in [0, T_+)} \left\|(u(\cdot, t), \partial_t u(\cdot, t))\right\|_{\dot{H}^{s_p} \times \dot{H}^{s_p-1}(\Rm^3)} < \infty
\]
must exist for all time $t>0$ and scatter in the positive time direction, unless the elliptic equation
\[
 - \Delta W = \phi (x) |W|^{p-1} W
\]
has a nonzero radial solution $W_0 \in C^2 (\Rm^3) \cap \dot{H}^{s_p} (\Rm^3)$.
\end{theorem}

\begin{remark}
 A similar result holds for the negative time direction as well, because the wave equation is time-reversible.
\end{remark}

\begin{remark}
 This is clear that if the elliptic equation does admit a radial $C^2$ solution $W_0 (x)$ in $\dot{H}^{s_p} (\Rm^3)$, then $u(x,t) = W_0 (x)$ is a solution to the wave equation (CP1) independent of tine $t$. Its critical Sobolev norm remains the same for all time but it definitely does not scatter. Therefore the condition about the elliptic equation in Theorem \ref{main theorem work} is not only a sufficient condition but also a necessary one. The solutions to the elliptic equation are usually called ground states.
\end{remark}

\begin{remark} \label{three types of solutions}
 In general, the global behaviour of a solution to a non-linear wave equation as $t \rightarrow T_+$ may be one of the following three cases
 \begin{itemize}
  \item[(I)] The solution scatters, i.e. it resembles the behaviour of a free wave\footnote{A free wave is a solution to the homogenous linear wave equation $\partial_t^2 u - \Delta u = 0$}. More precisely, $T_+ = +\infty$ and there exists a pair $(u_0^+, u_1^+) \in \dot{H}^{s_p} \times \dot{H}^{s_p-1} (\Rm^3)$, such that
  \[
   \lim_{t \rightarrow +\infty} \left\|\begin{pmatrix} u(t)\\ \partial_t u(t)\end{pmatrix} - {\mathbf S}_L(t) \begin{pmatrix} u_0^+\\ u_1^+ \end{pmatrix} \right\|_{\dot{H}^{s_p} \times \dot{H}^{s_p-1} (\Rm^3)} = 0.
  \]
  here ${\mathbf S}_L (t)$ is the linear wave propagation operator as in Definition \ref{wave propagation operator}.
  \item[(II)] The critical Sobolev norm of the solution is unbounded.
  \[
   \limsup_{t \rightarrow T_+} \left\|(u(\cdot, t), \partial_t u(\cdot, t))\right\|_{\dot{H}^{s_p} \times \dot{H}^{s_p-1} (\Rm^3)} = + \infty
  \]
  \item[(III)] The critical Sobolev norm of the solution is bounded but the solution does not scatter. One typical example is a ground state as mentioned above, if it exists.
 \end{itemize}
 Our main theorem claims that the case (III) is possible only if there is a perfect ground state.
\end{remark}

\subsection{Main Idea}

As in the special case $\phi(x) \equiv \pm 1$, the idea for the proof of the main theorem is to apply the compactness-rigidity argument (see also \cite{kenig, kenig1}). We start by giving a brief description about the compactness part of the argument.
\subsubsection{Compactness}
First of all, it suffices to verify that the statement $Sc(A)$ below is true for all $A > 0$, whenever an $\dot{H}^{s_p}$ ground state does not exist, in order to prove the main theorem.
\begin{statement} [$Sc(A)$]
If $u(x,t)$ is a radial solution of the non-linear wave equation (CP1) with a maximal lifespan $(-T_-, T_+)$, so that
\[
 \sup_{t \in [0,T_+)} \|(u(t), \partial_t u(t))\|_{\dot{H}^{s_p} \times \dot{H}^{s_p-1}(\Rm^3)} < A,
\]
then $T_+ = \infty$ and the solution scatters in the positive time direction.
\end{statement}

By the local theory given in section \ref{sec: local theory}, we know that $Sc(A)$ holds for small $A>0$. Our goal is to show $SC(A)$ holds for all $A>0$. If this were false, there would be a positive real number $M>0$, called the break-down point, so that $Sc(A)$ holds for all $A\leq M$ but fails for all $A>M$. As a result, we can pick up a sequence of non-scattering solutions $\{u_n\}_{n\in {\mathbb Z}^+}$, so that
\[
 \sup_{t \in [0,T_{n}^+)} \|(u_n (\cdot, t), \partial_t u_n (\cdot, t))\|_{\dot{H}^{s_p} \times \dot{H}^{s_p-1}} \searrow M.
\]
Here the notation $T_n^+$ represents the right-hand endpoint of the lifespan of $u_n$. The core of the compactness part is a limiting process: Possibly passing to a subsequence, we take a limit of the solutions $u_n$ mentioned above and finally obtain a ``critical element'' $u$, which is a solution to (CP1) defined for all $t \in \Rm$ and satisfies
\begin{itemize}
 \item $\displaystyle \sup_{t \in \Rm} \|(u (\cdot, t), \partial_t u (\cdot, t))\|_{\dot{H}^{s_p} \times \dot{H}^{s_p-1}(\Rm^3)} = M$.
 \item The set $\{(u(\cdot,t), \partial_t u(\cdot,t))| t\in \Rm \}$ is pre-compact in the space $\dot{H}^{s_p} \times \dot{H}^{s_p-1}(\Rm^3)$.
\end{itemize}
Among the key gradients of the compactness procedure are the profile decomposition and non-linear profiles associated to it.

\paragraph{The profile decomposition} Given a sequence of radial initial data $\{(u_{0,n}, u_{1,n})\}_{n \in {\mathbb Z}^+}$ which are uniformly bounded in the space $\dot{H}^{s_p} \times \dot{H}^{s_p-1} (\Rm^3)$, we can always find a subsequence of it, still denoted by $\{(u_{0,n}, u_{1,n})\}_{n \in {\mathbb Z}^+}$, a sequence of radial free waves, denoted by $\{V_j (x,t)\}_{j \in {\mathbb Z}^+}$, and a pair $(\lambda_{j,n}, t_{j,n}) \in \Rm^+ \times \Rm$ for each $(j,n) \in {\mathbb Z}^+ \times {\mathbb Z}^+$, such that
\begin{itemize}
 \item Given an integer $J>0$, we can write each pair of initial data in the subsequence into a sum of $J$ major components plus an error term:
\[
 (u_{0,n}, u_{1,n}) = \sum_{j=1}^J \left(V_{j,n}(\cdot, 0), \partial_t V_{j,n}(\cdot, 0)\right) + (w_{0,n}^J, w_{0,n}^J).
\]
Here $V_{j,n}$ is a modified version of $V_j$ via the application of a dilation and a time translation:
\[
 \left(V_{j,n}(x, t), \partial_t V_{j,n}(x, t) \right)= \left(\frac{1}{\lambda^{\frac{2}{p-1}}} V_j \left(\frac{x}{\lambda_{j,n}}, \frac{t-t_{j,n}}{\lambda_{j,n}}\right), \frac{1}{\lambda^{\frac{2}{p-1}+1}} \partial_t V_j \left(\frac{x}{\lambda_{j,n}}, \frac{t-t_{j,n}}{\lambda_{j,n}}\right)\right);
\]
and $(w_{0,n}^J, w_{1,n}^J)$ represents an error term that gradually becomes negligible as $J$ and $n$ grow.
 \item The sequences $\{(\lambda_{j,n}, t_{j,n})\}_{n \in {\mathbb Z}^+}$ and $\{(\lambda_{j',n}, t_{j',n})\}_{n \in {\mathbb Z}^+}$ are ``almost orthogonal'' for $j\neq j'$. More precisely we have
\[
 \lim_{n\rightarrow \infty} \left(\frac{\lambda_{j,n}}{\lambda_{j',n}} + \frac{\lambda_{j',n}}{\lambda_{j,n}} + \frac{|t_{j,n}-t_{j',n}|}{\lambda_{j,n}}\right) = +\infty.
\]
 \item We can also assume $\lambda_{j,n} \rightarrow \lambda_j \in \{0,1,\infty\}$ and $-t_{j,n}/\lambda_{j,n} \rightarrow t_j \in \Rm \cup \{\infty, -\infty\}$ as $n \rightarrow \infty$ for each fixed $j$, by possibly passing a subsequence and/or adjusting the free waves $\{V_j\}_{j\in {\mathbb Z}^+}$.
\end{itemize}
\paragraph{The nonlinear profiles} Let us first consider the case with a pure power-type nonlinearity. For each $j$ we can find a solution $U_j$ to (CP0), called a nonlinear profile, so that the function
\begin{equation} \label{def of Ujn}
 U_{j,n}(x,t) \doteq U_j \left(\frac{x}{\lambda_{j,n}}, \frac{t-t_{j,n}}{\lambda_{j,n}}\right)
\end{equation}
serves as a more and more accurate approximation of the solution to (CP0) with initial data $\left(V_{j,n}(\cdot, 0), \partial_t V_{j,n}(\cdot, 0)\right)$ when $n \rightarrow \infty$. We then add up these approximations for all $j\in {\mathbb Z}^+$ and finally obtain an approximation of $u_n$, thanks to the almost orthogonality. The fact that the equation (CP0) is invariant under dilations and time translations plays a crucial role in this argument. The same argument no longer works for the equation (CP1), since the presence of $\phi(x)$ prevents the application of dilations in this purpose. However, this difficulty can be overcome if we use nonlinear profiles that are not necessarily solutions to (CP1) but possibly solutions to other related equations instead. In fact, the solution to (CP1) with initial data $\left(V_{j,n}(\cdot, 0), \partial_t V_{j,n}(\cdot, 0)\right)$ can be approximated by a nonlinear profile $U_j$ as described below, up to a dilation and a time translation as shown in \eqref{def of Ujn}.
\begin{itemize}
 \item[I] {(Expanding Profile)} If $\lambda_j= \infty$, then the profile spreads out in the space as $n \rightarrow \infty$. Eventually a given compact set won't contain any significant part of the profile. The combination of this fact and our assumption $\lim_{|x| \rightarrow \infty} \phi (x) = \phi(\infty)$ implies that the nonlinear term $\phi(x) |u|^{p -1} u$ works in a similar way as $\phi(\infty) |u|^{p-1} u$. As a result, the nonlinear profile $U_j$ in this case is  a solution to the nonlinear wave equation $\partial_t^2 u - \Delta u = \phi (\infty) |u|^{p-1} u$.
 \item[II] {(Stable Profile)} If $\lambda_j =1$, then the profile approaches a stationary scale as $n \rightarrow \infty$. In this case the nonlinear profile $U_j$ is still a solution to (CP1).
 \item[III] {(Concentrating Profile)} If $\lambda_j = 0$, then the profile concentrates around the origin as $n \rightarrow \infty$. The nonlinear term $\phi(x) |u|^{p -1} u$ performs in almost the same way as $\phi(0) |u|^{p -1} u$. Therefore we can choose the nonlinear profile $U_j$ to be a solution of the semi-linear wave equation $\partial_t^2 u - \Delta u = \phi (0) |u|^{p-1} u$.
\end{itemize}

\subsubsection{Rigidity}

In this part we need to prove the non-existence of a critical element as mentioned above unless the equation (CP1) admits a nontrivial radial $C^2$ ground state in $\dot{H}^{s_p} (\Rm^3)$, i.e. the elliptic equation $-\Delta W(x) = \phi(x) |W(x)|^{p-1} W(x)$ has a radial solution in the space $C^2(\Rm^3) \cap \dot{H}^{s_p}(\Rm^3)$. A solution to this elliptic equation can be understood as a function defined for $(x,t) \in \Rm^3 \times \Rm$, although independent of $t$, which also solves (CP1). We usually call a solution of this type a ground state. Our proof of the rigidity part is straightforward: It turns out that any critical element must be exactly a ground state as mentioned above. The argument is similar to the one we used for the equation (CP0) and consists of three steps
\begin{itemize}
 \item [(I)] We first show that the critical element $u$ must be more regular than we have assumed. More precisely, it is in the space $\dot{H}^1 \times L^2 (\Rm^3 \setminus B(0,R))$ for all $R > 0$ and its behaviour near infinity is similar to that of $A/|x|$, where $A$ is a constant independent of $t$.
 \item [(II)] We then construct a solution $W$ to the equation $-\Delta W = \phi(x) |W|^{p-1} W$, whose behaviour near infinity is similar to that of $u$. Please note that this can be done even if the equation (CP1) does not admit a nontrivial $C^2$ ground state in $\dot{H}^{s_p}$. In this case the function $W$ is either outside the space $\dot{H}^{s_p} (\Rm^3)$ when $A\neq 0$, or identically zero when $A=0$. In fact, if the equation admit a ground state in $\dot{H}^{s_p} (\Rm^3)$, it can always be constructed via our method.
 \item [(II)] By applying the ``channel of energy'' method, we show that $u$ must be exactly the same as $W$. This gives a contradiction if the equation does not admit an $\dot{H}^{s_p}$ ground state. Because in this case $u(\cdot, t)$, which is the same as $W$ for each given $t \in \Rm$, is either outside the space $\dot{H}^{s_p} (\Rm^3)$ or identically zero.
\end{itemize}

\begin{remark} 
This argument works for a solution $u$ to (CP1) as long as 
\begin{itemize}
 \item The solution $u$ is radial and defined for all $t \in \Rm$;
 \item The set $\{(u(\cdot, t), \partial_t u(\cdot,t)): t \in \Rm\}$ is pre-compact in the space $\dot{H}^{s_p} \times \dot{H}^{s_p-1}(\Rm^3)$.
\end{itemize}
Any solution satisfying these conditions must be identically a ground state as mentioned above. 
\end{remark}

\subsection{Structure of this paper}
In section 2 we introduce notations, local theory, and already-known results as a preparation for the proof of the main theorem. The compactness part of the proof comes with two sections: In section 3 we make a review on the profile decomposition, introduce non-linear profiles and prove some of their properties. Next we carry on the compactness procedure and extract a critical element in section 4. The rigidity part of proof consists of three sections: We show the additional regularity of the critical element in section 5, then consider the solutions to the elliptic equation $-\Delta W = \phi(x) |W|^{p-1} W$ in section 6, and finally finish the proof in section 7 via the ``channel of energy'' method. Please note that the argument used in section 5 and section 7 is exactly the same one as the author used in \cite{shen2} to deal with the equation (CP0). Therefore we skip most details and merely give most important statements and ideas in these two sections.

\section{Preliminary Results}

\subsection{Notations}

\begin{definition}
 Throughout this paper we use the notation $F(u) = |u|^{p-1} u$.
\end{definition}

\begin{definition}
 We define $\Di_\lambda$ to be the dilation operator
 \[
  \Di_\lambda \left(u_0(x) ,u_1(x) \right) = \left(\frac{1}{\lambda^{3/2 - s_p}} u_0\left(\frac{x}{\lambda}\right),
\frac{1}{\lambda^{5/2 - s_p}} u_1\left(\frac{x}{\lambda}\right)\right);
 \]
 Here $x$ is the spatial variable of functions.
\end{definition}

\begin{definition} \label{wave propagation operator}
 Let ${\mathbf S}_L (t)$ be the linear wave propagation operator. More precisely, if $u$ is the solution to linear wave equation $\partial_t^2 u - \Delta u = 0$ with initial data $(u, \partial_t u)|_{t=0} = (u_0,u_1)$, then we define
\begin{align*}
 &{\mathbf S}_L (t_0) (u_0,u_1) = \left(u(\cdot, t_0), u_t (\cdot,t_0)\right),& &{\mathbf S}_L (t_0) \begin{pmatrix} u_0\\ u_1
 \end{pmatrix} =
 \begin{pmatrix} u(\cdot, t_0)\\ u_t (\cdot, t_0) \end{pmatrix}.&
\end{align*}
In addition, we use the notation $\mathbf{S}_{L,0} (u_0,u_1)$ for the first component $u(t_0)$ of the vector above.
\end{definition}
\begin{definition} Let $I$ be a time interval and $1 \leq q,r <\infty$. We define the space-time norm in the following way
\[
 \|u(x,t)\|_{L^q L^r (I \times \Rm^3)} = \left(\int_I \left(\int_{\Rm^3} |u(x,t)|^r\, dx \right)^{q/r}\,dt\right)^{1/q}.
\]
Similarly we have
\[
 \|u(x,t)\|_{L^\infty L^r (I \times \Rm^3)} = \esssup_{t \in I} \left(\int_{\Rm^3} |u(x,t)|^r\, dx\right)^{1/r}.
\]
\end{definition}

\subsection{Local Theory} \label{sec: local theory}

We start by the Strichartz estimates, as they are the basis of our local theory.

\begin{proposition} [Generalized Strichartz Inequalities] \label{strichartz}
(Please see Proposition 3.1 of \cite{strichartz}, here we use the Sobolev version in $\Rm^3$)
Let $2 \leq q_1,q_2 \leq \infty$, $2 \leq r_1, r_2 < \infty$ and $\rho_1, \rho_2, s \in \Rm$ with
\begin{align*}
 1/{q_i} + 1/{r_i} &\leq 1/2; \quad i=1,2;\\
 1/{q_1} + 3/{r_1} &= 3/2 - s + \rho_1;\\
 1/{q_2} + 3/{r_2} &= 1/2 + s + \rho_2.
\end{align*}
If $u$ is the solution of the following linear wave equation
\begin{equation} \label{linear wave equation}
\left\{\begin{array}{l} \partial_t^2 u - \Delta u = F (x,t), \,\,\,\,\, (x,t)\in \Rm^3 \times \Rm;\\
u |_{t=0} = u_0 \in \dot{H}^s (\Rm^3);\\
\partial_t u |_{t=0} = u_1 \in \dot{H}^{s-1}(\Rm^3);\end{array}\right.
\end{equation}
then for any time interval $I$ containing zero we have
\begin{align*}
 &\sup_{t \in I} \|(u(\cdot, t), \partial_t u(\cdot, t))\|_{\displaystyle \dot{H}^{s} \times \dot{H}^{s-1}(\Rm^3)} +
 \|D_x^{\rho_1} u\|_{\displaystyle L^{q_1} L^{r_1} (I\times \Rm^3)}\\
 &\quad\leq C \left( \|(u_0,u_1)\|_{\displaystyle \dot{H}^{s} \times \dot{H}^{s-1}(\Rm^3)} +
 \|D_x^{-\rho_2} F(x,t)\|_{\displaystyle L^{\bar{q}_2} L^{\bar{r}_2}(I\times \Rm^3)} \right).
\end{align*}
The constant $C$ does not depend on the time interval $I$.
\end{proposition}

\begin{definition} Let $I$ be a time interval. We define the following norms
\begin{align*}
 \|u(x,t)\|_{Y(I)} &= \|u(x,t)\|_{L^\frac{2p}{1+s_p} L^\frac{2p}{2-s_p}(I\times \Rm^3)};\\
 \|v(x,t)\|_{Z(I)} &= \|u(x,t)\|_{L^\frac{2}{1+s_p} L^\frac{2}{2-s_p}(I\times \Rm^3)}.
\end{align*}
\end{definition}

\begin{definition}
It may be necessary to use the notation
\[
\|(u_0,u_1)\|_H = \|(u_0,u_1)\|_{\dot{H}^{s_p} \times \dot{H}^{s_p -1}(\Rm^3)}
\]
in order to save space. If $V$ is a free wave, then the norm $\|(V(\cdot,t), \partial_t V(\cdot,t)\|_H$ is independent of $t$. Thus we may use the notation $\|V\|_H$ instead for simplicity.
\end{definition}

\paragraph{The fixed-point argument} If $u$ is a solution to \eqref{linear wave equation} on a time interval $I$ containing $0$, then we have the Strichartz estimates
\[
 \sup_{t \in I} \|(u(t), \partial_t u(t))\|_{H} + \|u\|_{Y(I)} \leq C \left[\|(u_0,u_1)\|_H + \|F\|_{Z(I)}\right].
\]
Combining this with the inequalities
\begin{align*}
 \|\phi F(u)\|_{Z(I)} & \leq \|u\|_{Y(I)}^p;\\
 \|\phi F(u_1) - \phi F(u_2)\|_{Z(I)} & \leq C_p \|u_1 -u_2\|_{Y(I)}\left[\|u_1\|_{Y(I)}^{p-1} + \|u_2\|_{Y(I)}^{p-1}\right];
\end{align*}
and applying a fixed-point argument, we obtain a local theory as given in the rest of this subsection. Since our argument is similar to those in a lot of earlier works, we only give important statements but omit most of the proof here. Please see, for instance, \cite{kenig, ls} for more details.

\begin{definition} [Solutions]
We say $u(t)$ is a solution of (CP1) on the time interval $I$, if $(u(t),\partial_t u(t)) \in C(I;{\dot{H}^{s_p}}\times{\dot{H}^{s_p-1}}(\Rm^3))$, with a finite norm $\|u\|_{Y(J)}$ for any bounded closed interval $J \subseteq I$ so that the integral equation
\[
 u(t) = {\mathbf S}_{L,0} (t)(u_0,u_1) + \int_0^t \frac{\sin ((t-\tau)\sqrt{-\Delta})}{\sqrt{-\Delta}} F(u(\tau)) d\tau
\]
holds for all time $t \in I$.
\end{definition}

\begin{theorem} [Local solution]
For any initial data $(u_0,u_1) \in \dot{H}^{s_p} \times \dot{H}^{s_p-1}$, there is a maximal interval $(-T_{-}(u_0,u_1), T_{+}(u_0,u_1))$
in which the equation has a unique solution.
\end{theorem}

\begin{theorem} [Scattering with small data] \label{Scattering with small data}
There exists $\delta = \delta(p) > 0$ such that if the norm of the initial data $\|(u_0,u_1)\|_{\dot{H}^{s_p} \times \dot{H}^{s_p-1}} < \delta$, then the Cauchy problem (CP1) has a global-in-time solution $u$ with $\|u\|_{Y(-\infty,+\infty)} \leq C_p \|(u_0,u_1)\|_{\dot{H}^{s_p} \times \dot{H}^{s_p-1}}$. Here both the constants $\delta(p)$ and $C_p$ can be chosen independent of the coefficient function $|\phi(x)| \leq 1$.
\end{theorem}

\begin{corollary} \label{small data preserve}
 There exists a function $\eta: \Rm^+ \rightarrow \Rm^+$, such that if $\|(u_0,u_1)\|_{\dot{H}^{s_p} \times \dot{H}^{s_p-1}} \geq C_1 > 0$, then the solution $u$ to (CP1) with the initial data $(u_0,u_1)$ satisfies
 \[
  \inf_{t \in (-T_-, T_+)} \|(u(\cdot,t),\partial_t u(\cdot, t)\|_{\dot{H}^{s_p} \times \dot{H}^{s_p-1}(\Rm^3)} \geq \eta(C_1),
 \]
\end{corollary}

\begin{lemma} [Standard finite blow-up criterion]
If $T_{+} < \infty$, 
then $\|u\|_{Y([0,T_{+}))} = \infty$.
\end{lemma}

\begin{theorem} [Perturbation theory] \label{perturbation theory in Ysp}
Fix $3<p<5$. Let $M$ be a positive constant. There exists a constant $\eps_0 = \eps_0 (M, p)>0$, such that if an approximation solution $\tilde{u}$ defined on $\Rm^3 \times I$ ($0\in I$) and a pair of initial data $(u_0,u_1) \in \dot{H}^{s_p} \times \dot{H}^{s_p-1}$ satisfy
\begin{align*}
 & (\partial_t^2 - \Delta) \tilde{u} - \phi F(\tilde{u}) = e(x,t), \qquad (x,t) \in \Rm^3 \times I;\\
 &\|\tilde{u}\|_{Y(I)} < M;\qquad\qquad \|(\tilde{u}(\cdot, 0),\partial_t \tilde{u}(\cdot, 0))\|_{\dot{H}^{s_p}\times \dot{H}^{s_p-1}} < \infty;\\
 &\eps \doteq \|e(x,t)\|_{Z(I)}+ \|{\mathbf S}_{L,0} (t)(u_0-\tilde{u}(\cdot, 0),u_1 - \partial_t \tilde{u}(\cdot, 0))\|_{Y(I)} < \eps_0;
\end{align*}
then there exists a solution $u(x,t)$ of (CP1) defined in the interval $I$ with the initial data $(u_0,u_1)$ and satisfying
\[
 \|u(x,t) - \tilde{u}(x,t)\|_{Y(I)} < C(M, p) \eps;
\]
\[
 \sup_{t \in I} \left\|\begin{pmatrix} u(t)\\ \partial_t u(t)\end{pmatrix} - \begin{pmatrix} \tilde{u}(t)\\ \partial_t \tilde{u}(t)\end{pmatrix} - {\mathbf S}_L (t)\begin{pmatrix} u_0 - \tilde{u}(0)\\ u_1 -\partial_t \tilde{u}(0)\end{pmatrix}
 \right\|_{\dot{H}^{s_p} \times \dot{H}^{s_p-1}} < C(M, p)\eps.
\]
\end{theorem}

\begin{proof}
Let us first prove the perturbation theory when $M$ is sufficiently small. Let $I_1$ be the maximal lifespan of the solution $u(x,t)$ to the Cauchy problem (CP1) with the given initial data $(u_0,u_1)$ and assume $[-T_1,T_2] \subseteq I\cap I_1$. By the Strichartz estimates, we have
\begin{align*}
\|\tilde{u}-u\|_{Y ([-T_1,T_2])} &\leq \|{\mathbf S}_{L,0} (t)(u_0-\tilde{u}(0), u_1-\tilde{u}(0))\|_{Y ([-T_1,T_2])}\\
&\qquad \qquad\qquad + C_p \|e + \phi F (\tilde{u})- \phi F(u)\|_{Z([-T_1,T_2])}\\
&\leq \eps + C_p \|e\|_{Z([-T_1,T_2])} + C_p \|F(\tilde{u}) - F(u)\|_{Z([-T_1,T_2])}\\
&\leq \eps + C_p \eps + C_p \| \tilde{u}-u\|_{Y([-T_1,T_2])} \left(\|\tilde{u}\|_{Y([-T_1,T_2])}^{p-1} +
\|\tilde{u}-u\|_{Y([-T_1,T_2])}^{p-1}\right)\\
&\leq C_p \eps + C_p \|\tilde{u}-u\|_{Y ([-T_1,T_2])} \left(M^{p-1} + \|\tilde{u}-u\|_{Y([-T_1,T_2])}^{p-1}\right).
\end{align*}
Here the notation $C_p$ may represent different constants at different places but all these constants depend solely on $p$. By a continuity argument in $T_1, T_2$, there exist $M_0 = M_0 (p)>0$ and $\eps_0 = \eps_0 (p) >0$, such that if $M \leq M_0$ and $\eps < \eps_0$, we have
\[
 \|\tilde{u}-u\|_{Y ([-T_1,T_2])} \leq C_p \eps.
\]
Observing that the estimate above is independent of the time interval $[-T_1,T_2]$, we are actually able to conclude $I \subseteq I_1$ by the standard blow-up criterion and obtain
\[
 \|\tilde{u}-u\|_{Y (I)} \leq C_p \eps.
\]
In addition, by the Strichartz estimate we have
\begin{align*}
 \sup_{t \in I} & \left\|\begin{pmatrix} u(t)\\ \partial_t u(t)\end{pmatrix} - \begin{pmatrix} \tilde{u}(t)\\ \partial_t \tilde{u}(t)\end{pmatrix}
 - {\mathbf S}_L (t) \begin{pmatrix} u_0 - \tilde{u}(0)\\ u_1 -\partial_t \tilde{u}(0)\end{pmatrix} \right\|_{\dot{H}^{s_p} \times \dot{H}^{s_p-1}}\\
 &\leq C_p\| \phi F(u) - \phi F(\tilde{u}) - e\|_{Z(I)}\\
 &\leq C_p\left(\|e\|_{Z(I)} + \|F(u) - F( \tilde{u})\|_{Z(I)}\right)\\
 &\leq C_p\left[\eps + \|u -\tilde{u}\|_{Y(I)}\left(\|\tilde{u}\|_{Y(I)}^{p-1} + \|u-\tilde{u}\|_{Y (I)}^{p-1}\right)\right]\\
 &\leq C_p \eps.
\end{align*}
This finishes the proof as $M$ is sufficiently small. To deal with the general case, we can separate the time interval $I$ into finite number of subintervals $\{I_j\}$, so that $\|\tilde{u}\|_{Y(I_j)} < M_0$, and then iterate our argument above.
\end{proof}

\begin{remark} \label{lifespan lower bound for compact set}
If $K$ is a compact subset of the space $\dot{H}^{s_p} \times \dot{H}^{s_p-1}(\Rm^3)$, then there exists $T = T(K) > 0$ such that for any $(u_0,u_1) \in K$, $T_{+}(u_0,u_1) > T(K)$. This is a direct corollary of the perturbation theory.
\end{remark}

\subsection{Known Results with a Constant Coefficient}

In this subsection we make a review on the already-known results concerning radial solutions to the equation
\begin{equation} \label{equation with constant phi}
 \left\{\begin{array}{l}
  \partial_t^2 u - \Delta u = c |u|^{p-1} u; \\
  u|_{t=0} = u_0 \in \dot{H}^{s_p}  (\Rm^3); \\
  \partial_t u|_{t =0} = u_1 \in \dot{H}^{s_p-1} (\Rm^3);
 \end{array}\right.
\end{equation}
Here $c$ is a constant. The case with $c = \pm 1$, namely the equation (CP0), has been discussed in the author's previous work \cite{shen2}, whose main result has been mentioned in the introduction as Theorem \ref{theorem with phi pm 1}. If $u$ is a solution to \eqref{equation with constant phi} with $c \neq \pm1$, then $|c|^{1/(p-1)} u$ is a solution to (CP0); and vice versa. This transformation immediately gives

\begin{proposition} \label{theorem with constant phi}
Let $u$ be a radial solution to the equation \eqref{equation with constant phi} with a maximal lifespan $(-T_-, T_+)$ and a uniform boundedness condition on the critical Sobolev norm
\[
 \sup_{t \in [0, T_+)} \left\|(u(\cdot, t), \partial_t u(\cdot, t))\right\|_{\dot{H}^{s_p} \times \dot{H}^{s_p-1}(\Rm^3)} < \infty.
\]
Then $T_+ = \infty$ and $u$ scatters in the positive time direction.
\end{proposition}

\subsection{Properties of Radial $\dot{H}^{s}$ Functions}

\begin{lemma} (Please see lemma 3.2 of \cite{km}) \label{pointwise radial Hs}
Let $1/2<s<3/2$. Any radial $\dot{H}^s (\Rm^3)$ function $u$ satisfies the inequality
\[
 \left|u(x)\right| \lesssim_s \frac{\|u\|_{\dot{H}^s (\Rm^3)}}{|x|^{\frac{3}{2} -s}}.
\]
\end{lemma}
\begin{remark}
This actually means that a radial $\dot{H}^s$ function is uniformly continuous in $\Rm^3 \setminus B(0,R)$ if $R > 0$.
\end{remark}

\begin{lemma} \label{uniform decay for compact set}
 Let $K$ be a compact subset of $\dot{H}^{s}(\Rm^3)$, $1/2 < s < 3/2$. Then we have
\begin{align*}
 &\sup_{|x|>R, u \in K} |x|^{\frac{3}{2}-s} |u(x)| \rightarrow 0, \quad \hbox{as}\; R \rightarrow +\infty;\\
 &\sup_{|x|<r, u \in K} |x|^{\frac{3}{2}-s} |u(x)| \rightarrow 0, \quad \hbox{as}\; r \rightarrow 0^+
\end{align*}
\end{lemma}
\begin{proof}
 A Combination of the compactness with Lemma \ref{pointwise radial Hs} shows that it suffices to prove this lemma when $K$ contains a single element. The proof of this special case has been given in Appendix of \cite{shen2}.
\end{proof}

\section{Profile Decomposition} \label{sec: profile decomposition3}
\subsection{Linear Profile Decomposition}
\begin{theorem} [Profile Decomposition] \label{profile decomposition}
Let $A$ be a constant. Given a sequence of radial initial data $\{(u_{0,n}, u_{1,n})\}_{n \in {\mathbb Z}^+}$ so that $\|(u_{0,n}, u_{1,n})\|_H \leq A$, there exist a subsequence of it, still denoted by $(u_{0,n}, u_{1,n})$; a sequence of radial free waves $V_j(x,t) = {\mathbf S}_{L,0} (t) (v_{j,0}, v_{j,1}), j\in {\mathbb Z}^+$; a pair $(\lambda_{j,n}, t_{j,n}) \in \Rm^+ \times \Rm$ for each $(j,n) \in {\mathbb Z}^+ \times {\mathbb Z}^+$; such that
\begin{itemize}
 \item[(i)] Given a positive integer $J$, each pair of initial data in the subsequence can be expressed as a sum of the first $J$ major components plus an error term
 \begin{align*}
 (u_{0,n}, u_{1,n}) & = \sum_{j =1}^J \left(\frac{1}{\lambda_{j,n}^{3/2 - s_p}} V_j\left(\frac{\cdot}{\lambda_{j,n}}, \frac{-t_{j,n}}{\lambda_{j,n}}\right),
 \frac{1}{\lambda_{j,n}^{5/2 -s_p}} \partial_t V_j\left(\frac{\cdot}{\lambda_{j,n}}, \frac{-t_{j,n}}{\lambda_{j,n}}\right)\right) + (w_{0,n}^J, w_{1,n}^J)\\
& = \sum_{j=1}^J {\mathbf S}_L (-t_{j,n}) \Di_{\lambda_{j,n}} (v_{0,j}, v_{1,j}) +  (w_{0,n}^J, w_{1,n}^J);
\end{align*}
\item[(ii)] If $j \neq j'$, then the sequences $\{(\lambda_{j,n}, t_{j,n})\}_{n \in {\mathbb Z}^+}$ and $\{(\lambda_{j',n}, t_{j',n})\}_{n \in {\mathbb Z}^+}$ are ``almost orthogonal'', i.e. we have the limit
\[
 \lim_{n \rightarrow \infty} \left(\frac{\lambda_{j',n}}{\lambda_{j,n}} + \frac{\lambda_{j,n}}{\lambda_{j',n}} + \frac{|t_{j,n}-t_{j',n}|}{\lambda_{j,n}}\right) = +\infty.
\]
\item[(iii)] $\displaystyle \limsup_{n \rightarrow \infty} \left\|{\mathbf S}_L (t)(w_{0,n}^J, w_{1,n}^J)\right\|_{Y(\Rm)} \rightarrow 0$ as $J \rightarrow \infty$.
\item[(iv)] For each given $J \geq 1$, we have
\[
 \|(u_{0,n}, u_{1,n})\|_H^2 = \sum_{j=1}^J \|V_j\|_H^2 + \left\|(w_{0,n}^J, w_{1,n}^J)\right\|_H^2 + o_{J,n}(1).
\]
Here $o_{J,n}(1) \rightarrow 0$ as $n \rightarrow \infty$.
\item[(v)] We have the limits $\lambda_{j,n} \rightarrow \lambda_j \in \{0,1,\infty\}$ and $-t_{j,n}/\lambda_{j,n} \rightarrow t_j \in [-\infty, \infty]$ as $n \rightarrow \infty$ for each $j$.
\end{itemize}
\end{theorem}

Please see \cite{bahouri} for the proof. There are a few remarks.
\begin{itemize}
 \item This original paper deals with the energy critical case $p=5$. But the same argument works for all $3 < p <5$ as well.
 \item The original paper works for non-radial initial data as well. In this work we only consider the radial case.
 \item The original theorem is proved under an additional assumption labelled (1.6) there. But this condition can be eliminated according to Remark 5 on Page 159 of that paper. The elimination of this condition also implies that $\lambda_j$, the limit of the sequence $\lambda_{j,n}$ as $n \rightarrow \infty$, may converge to $1$ or $+\infty$, besides $0$,  as given in part (v) above.
\end{itemize}

We need to prove a few lemmata before the introduction of the non-linear profiles.

\begin{lemma} \label{almost orthogonality 4}
If $j \neq j'$, then we have the almost orthogonality
 \[
  \lim_{n \rightarrow \infty} \left\langle \mathbf{S}_L (-t_{j,n}) \Di_{\lambda_{j,n}} (v_{0,j}, v_{1,j}),  \mathbf{S}_L (-t_{j',n}) \Di_{\lambda_{j',n}} (v_{0,j'}, v_{1,j'}) \right\rangle_H = 0.
 \]
\end{lemma}
\begin{proof}
We rewrite the inner product into
\begin{align*}
  &\left\langle \mathbf{S}_L (-t_{j,n}) \Di_{\lambda_{j,n}} (v_{0,j}, v_{1,j}),  \mathbf{S}_L (-t_{j',n}) \Di_{\lambda_{j',n}} (v_{0,j'}, v_{1,j'}) \right\rangle\\
 =& \left\langle \Di_{\lambda_{j,n}/\lambda_{j',n}} (v_{0,j}, v_{1,j}), {\mathbf S}_L \left(\frac{t_{j,n}-t_{j',n}}{\lambda_{j',n}}\right) (v_{0,j'}, v_{1,j'}) \right\rangle.
\end{align*}
We can immediately finish the proof by the almost orthogonal condition (ii) and basic Fourier analysis.
\end{proof}

\begin{lemma}\label{weak convergence}
Let $\{(w_{0,n}, w_{1,n})\}_{n\in {\mathbb Z}^+}$ be a bounded sequence in the space $\dot{H}^{s_p} \times \dot{H}^{s_p-1}$, i.e. $\|(w_{0,n}, w_{1,n})\|_{\dot{H}^{s_p} \times \dot{H}^{s_p-1}} \leq A$ so that $\|{\mathbf S}_{L,0} (t)(w_{0,n}, w_{1,n})\|_{Y(\Rm)} \rightarrow 0$. Then we have the weak limit $(w_{0,n}, w_{1,n}) \rightharpoonup 0$ in $\dot{H}^{s_p} \times \dot{H}^{s_p-1}$.
\end{lemma}
\begin{proof}
 If the weak limit $(w_{0,n}, w_{1,n}) \rightharpoonup 0$ were not true, we could assume $(w_{0,n}, w_{1,n}) \rightharpoonup (w_0,w_1) \neq 0$ in $\dot{H}^{s_p} \times \dot{H}^{s_p-1}$ by possibly passing to a subsequence. Because the map $(u_0,u_1) \rightarrow {\mathbf S}_{L,0} (t) (u_0,u_1)$ is a bounded linear operator from the space $\dot{H}^{s_p} \times \dot{H}^{s_p-1}$ to $Y(\Rm)$ by the Strichartz estimates, we also have a weak limit ${\mathbf S}_{L,0} (t)(w_{0,n}, w_{1,n}) \rightharpoonup {\mathbf S}_{L,0} (t)(w_0,w_1)$ in the space $Y(\Rm)$. On the other hand, the same sequence ${\mathbf S}_{L,0} (t)(w_{0,n}, w_{1,n})$ has a strong limit zero in the space $Y(\Rm)$ by the conditions given. As a result, we have ${\mathbf S}_{L,0} (t)(w_0,w_1) = 0 \Longrightarrow (w_0,w_1) = 0$. This is a contradiction.
\end{proof}

\begin{lemma}\label{almost orthogonal of w}
 Assume $\|(w_{0,n}, w_{1,n})\|_{\dot{H}^{s_p} \times \dot{H}^{s_p-1}} \leq A$ and $\|{\mathbf S}_{L,0} (t)(w_{0,n}, w_{1,n})\|_{Y(\Rm)} \rightarrow 0$. Let $I $ be a closed time interval and $(U_0(x,t), U_1(x,t)) \in C(I; \dot{H}^{s_p} \times \dot{H}^{s_p-1})$. If $I$ contains a neighbourhood of $\infty$ or $-\infty$, we also assume
 \[
  \lim_{t \rightarrow \pm \infty} \left\|\begin{pmatrix} U_0(\cdot,t)\\ U_1(\cdot,t)\end{pmatrix}- {\mathbf S}_L (t) \begin{pmatrix} u_0^\pm \\ u_1^\pm \end{pmatrix}\right\|_{\dot{H}^{s_p} \times \dot{H}^{s_p-1}} = 0
 \]
 for some pair(s) $(u_0^\pm ,u_1^\pm) \in \dot{H}^{s_p} \times \dot{H}^{s_p-1}$. Then for any two sequences $\{\lambda_n: \lambda_n >0\}_{n \in {\mathbb Z}^+}$ and $\{t_n: t_n \in I\}_{n \in {\mathbb Z}^+}$, we always have the limit
\[
 \langle \Di_{\lambda_n} (U_0(\cdot, t_n), U_1(\cdot, t_n)), (w_{0,n}, w_{1,n})\rangle_{\dot{H}^{s_p} \times \dot{H}^{s_p-1}} \rightarrow 0.
\]
\end{lemma}
\begin{proof}
 First of all, we can rewrite the pairing into
\begin{align*}
 & \langle \Di_{\lambda_n} (U_0(\cdot, t_n), U_1(\cdot, t_n)), (w_{0,n}, w_{1,n})\rangle_{\dot{H}^{s_p} \times \dot{H}^{s_p-1}} \\
  = & \langle  (U_0(\cdot, t_n), U_1(\cdot, t_n)), \Di_{1/\lambda_n} (w_{0,n}, w_{1,n})\rangle_{\dot{H}^{s_p} \times \dot{H}^{s_p-1}}\\
  = & \langle  {\mathbf S}_L (-t_n) (U_0(\cdot, t_n), U_1(\cdot, t_n)), {\mathbf S}_L (-t_n) \Di_{1/\lambda_n} (w_{0,n}, w_{1,n})\rangle_{\dot{H}^{s_p} \times \dot{H}^{s_p-1}}.
\end{align*}
According to the conditions given, we also have
\begin{itemize}
 \item The set $\{{\mathbf S}_L (-t) (U_0(\cdot, t), U_1(\cdot, t))| t \in I\}$ is pre-compact in $\dot{H}^{s_p} \times \dot{H}^{s_p-1}$.
 \item The sequence ${\mathbf S}_L (-t_n) \Di_{1/\lambda_n} (w_{0,n}, w_{1,n})$ converges weakly to $0$ in the space $\dot{H}^{s_p} \times \dot{H}^{s_p-1}$, because of Lemma \ref{weak convergence} and
 \begin{align*}
  \left\|{\mathbf S}_L (-t_n) \Di_{1/\lambda_n} (w_{0,n}, w_{1,n})\right\|_{\dot{H}^{s_p} \times \dot{H}^{s_p-1}} & = \left\|(w_{0,n}, w_{1,n})\right\|_{\dot{H}^{s_p} \times \dot{H}^{s_p-1}} \leq A;\\
  \left\|{\mathbf S}_{L,0} (t) {\mathbf S}_L (-t_n) \Di_{1/\lambda_n} (w_{0,n}, w_{1,n})\right\|_{Y(\Rm)} & = \left\|{\mathbf S}_{L,0} (t)(w_{0,n}, w_{1,n})\right\|_{Y(\Rm)} \rightarrow 0.
 \end{align*}
\end{itemize}
Therefore the pairing converges to zero.
\end{proof}
\subsection{Nonlinear Profiles}

In this subsection we introduce the nonlinear profiles and prove some properties of them.

\begin{definition} [A nonlinear profile]
 Fix $\tilde{\phi}$ to be either the function $\phi$ or a constant function $c$. Let $V(x,t) = {\mathbf S}_{L,0} (t) (v_0,v_1)$ be a free wave and $\tilde{t} \in [-\infty,\infty]$ be a time. We say that $U(x,t)$ is a nonlinear profile associated to $(V,\tilde{\phi}, \tilde{t})$ if $U(x,t)$ is a solution to the nonlinear wave equation
\begin{equation} \label{nonlinear profile1}
 \partial_t^2 u - \Delta u = \tilde{\phi} F(u)
\end{equation}
with a maximal timespan $I$ so that $I$ contains a neighbourhood\footnote{A neighbourhood of infinity is $(M, +\infty)$, if $\tilde{t} = +\infty$; or $(-\infty, M)$, if $\tilde{t} = - \infty$.} of $\tilde{t}$ and
\[
 \lim_{t \rightarrow \tilde{t}} \|(U(\cdot,t),\partial_t U (\cdot,t))- (V(\cdot,t), \partial_t V(\cdot,t))\|_{\dot{H}^{s_p} \times \dot{H}^{s_p-1}} = 0.
\]
\end{definition}
\begin{remark}
Given a triple $(V, \tilde{\phi}, \tilde{t})$ as above, one can show there is always a unique nonlinear profile. Please see Remark 2.13 in \cite{kenig1} for the idea of proof. In particular, if $\tilde{t}$ is finite, then the nonlinear profile $U$ is simply the solution to the equation \eqref{nonlinear profile1} with the initial data $(U(\cdot, \tilde{t}), \partial_t U(\cdot, \tilde{t})) = (V(\cdot, \tilde{t}), \partial_t V(\cdot, \tilde{t}))$. We will also use the fact that the nonlinear profile automatically scatters in the positive time direction if $\tilde{t} = + \infty$.
\end{remark}


\begin{definition} [Nonlinear Profiles] \label{def of nonlinear profiles}
 For each linear profile $V_j$ in a profile decomposition as given in Theorem \ref{profile decomposition}, we define $U_j$ to be the nonlinear profile associated to $(V_j, \phi_j, t_j)$. Here the coefficient function $\phi_j$ is chosen according to the value of $\lambda_j$:
 \begin{itemize}
 \item If $\lambda_j = 0$, then we choose $\phi_j(x) \equiv \phi(0)$;
 \item If $\lambda_j = 1$, then we choose $\phi_j(x) = \phi (x)$;
 \item If $\lambda_j = \infty$, then we choose $\phi_j(x) \equiv \phi(\infty)$.
 \end{itemize}
In either case, we use the notation $I_j$ for the maximal lifespan of $U_j$ and define
\[
 U_{j,n}(x,t) \doteq \frac{1}{\lambda_{j,n}^{3/2-s_p}} U_j \left(\frac{x}{\lambda_{j,n}}, \frac{t - t_{j,n}}{\lambda_{j,n}}\right).
\]
\end{definition}

\begin{remark} \label{nonlinear and linear profiles are close}
By the definition of nonlinear profile, for each $j$ we have the limit
\[
 \displaystyle \lim_{n \rightarrow \infty} \|(U_{j,n}(\cdot, 0), \partial_t U_{j,n}(\cdot,0)) - \mathbf{S}_L (-t_{j,n}) \Di_{\lambda_{j,n}} (v_{0,j}, v_{1,j})\|_{\dot{H}^{s_p} \times \dot{H}^{s_p-1}} = 0.
\]
\end{remark}

\begin{lemma} \label{almost orthogonality 3}
 If $j\neq j'$, then we have the following almost orthogonality.
 \[
  \lim_{n \rightarrow \infty} \left\langle (U_{j,n}(\cdot, 0), \partial_t U_{j,n}(\cdot,0)), (U_{j',n}(\cdot, 0), \partial_t U_{j',n}(\cdot,0)) \right\rangle_{\dot{H}^{s_p} \times \dot{H}^{s_p-1}} = 0.
 \]
\end{lemma}
\begin{proof} This is a direct corollary of Remark \ref{nonlinear and linear profiles are close} and Lemma \ref{almost orthogonality 4}.
\end{proof}

\begin{lemma} \label{general almost orthogonality}
Assume $\|\tilde{U}_j\|_{Y(I'_j)} < \infty$ for $j=1,2$.  Let $\{(\lambda_{1,n}, t_{1,n})\}_{n \in {\mathbb Z}^+}$ and $\{(\lambda_{2,n}, t_{2,n})\}_{n \in {\mathbb Z}^+}$  be two ``almost orthogonal'' sequences of pairs, i.e.
\[
 \lim_{n \rightarrow +\infty} \left(\frac{\lambda_{2,n}}{\lambda_{1,n}}+ \frac{\lambda_{1,n}}{\lambda_{2,n}} + \frac{|t_{1,n} - t_{2,n}|}{\lambda_{1,n}}\right) = + \infty.
\]
If $\{J_n\}$ is a sequence of time intervals, such that $J_n \subseteq (t_{1,n} + \lambda_{1,n}I'_1)\cap (t_{2,n} + \lambda_{2,n}I'_2)$ holds for all sufficiently large positive integers $n$, then we have
\[
 N(n) \doteq \left\|\tilde{U}_{1,n} \tilde{U}_{2,n} \right\|_{L_t^\frac{p}{1+s_p} L_x^\frac{p}{2-s_p} (J_n \times \Rm^3)} \rightarrow 0, \qquad \hbox{as}\;\; n \rightarrow \infty.
\]
Here $\tilde{U}_{j,n}$ is defined as usual
\[
 \tilde{U}_{j,n}(x,t) = \frac{1}{\lambda_{j,n}^{3/2-s_p}} \tilde{U}_j \left(\frac{x}{\lambda_{j,n}}, \frac{t - t_{j,n}}{\lambda_{j,n}}\right).
\]
\end{lemma}

\begin{proof} (See also Lemma 2.7 in \cite{orthogonality}) First of all, we only need to consider the special case with $I'_j = \Rm$ and $J_n = \Rm$ for all $j,n \in {\mathbb Z}^+$. Otherwise we can extend the domain of the functions by defining $\tilde{U}_j (x,t) = 0$ for all $t \notin I'_j$. Observing the continuity of the map
\[
 \Phi: Y(\Rm) \times Y(\Rm) \rightarrow l^\infty,\quad \Phi (\tilde{U}_1,\tilde{U}_2) = \left\{\left\|\tilde{U}_{1,n} \tilde{U}_{2,n}\right\|_{L_t^\frac{p}{1+s_p} L_x^\frac{p}{2-s_p} (\Rm \times \Rm^3)}\right\}_{n \in {\mathbb Z}^+};
\]
we can also assume, without loss of generality, that
\begin{align*}
 &\left|\tilde{U}_j (x,t)\right| \leq M_j, \; \hbox{for any}\; (x,t)\in \Rm^3 \times \Rm;& &\hbox{Supp} (\tilde{U}_j) \subseteq \{(x,t): |x|,|t| < R_j\} &
\end{align*}
for some constants $M_j$, $R_j$ and $j=1,2$, since the functions satisfying these conditions are dense in the space $Y (\Rm)$.  If the conclusion were false, we would find a sequence $n_1 < n_2 < n_3 < \cdots$ and a positive constant $\eps_0$ such that $ N(n_k) \geq \eps_0$. There are three cases
\begin{itemize}
 \item [(I)] $\limsup_{k \rightarrow \infty} \lambda_{1, n_k}/ \lambda_{2,n_k} = \infty$. First of all, the product $\tilde{U}_{1,n_k} \tilde{U}_{2,n_k}$ is supported in the $(3+1)$-dimensional circular cylinder centred at $(0, t_{2,n_k})$ with radius $\lambda_{2,n_k} R_2$ and height $2\lambda_{2,n_k} R_2$ because $\tilde{U}_{2,n_k}$ is supported in this cylinder. In addition we have
 \[
  \left|\tilde{U}_{1,n_k} \tilde{U}_{2,n_k}\right| \leq \lambda_{1,n_k}^{-3/2+s_p} \lambda_{2,n_k}^{-3/2+s_p} M_1 M_2.
 \]
A basic computation shows
\[
N(n_k) = \left\|\tilde{U}_{1,n_k} \tilde{U}_{2,n_k} \right\|_{L_t^\frac{p}{1+s_p} L_x^\frac{p}{2-s_p} (\Rm \times \Rm^3)} \leq C(p) M_1 M_2 R_2^{3-2s_p} \left(\frac{\lambda_{2, n_k}}{\lambda_{1,n_k}}\right)^{3/2-s_p}.
\]
This upper bound tends to zero as $\lambda_{1, n_k}/ \lambda_{2,n_k} \rightarrow \infty$. Thus we have a contradiction.
\item [(II)] $\limsup_{k \rightarrow \infty} \lambda_{2, n_k}/ \lambda_{1,n_k} = \infty$. This can be handled in the same way as case (I).
\item [(III)] $\lambda_{1, n_k} \simeq \lambda_{2, n_k}$. By the ``almost orthogonality'' of the sequences of pairs, we also have
\[
 \frac{|t_{1,n_k} - t_{2,n_k}|}{\lambda_{1,n_k}} \rightarrow \infty.
\]
This implies $\hbox{Supp}(\tilde{U}_{1,n_k}) \cap \hbox{Supp}(\tilde{U}_{2,n_k}) = \emptyset$ when $k$ is sufficiently large thus gives a contradiction.
\end{itemize}
\end{proof}

\begin{lemma} \label{uniform estimate in J}
 Assume $I'_j \subseteq I_j$ with $\|U_j(x,t)\|_{Y(I'_j)} < \infty$. Let $\{J_n\}$ be a sequence of time intervals, so that given $J\in {\mathbb Z}^+$ we have $\displaystyle J_n \subseteq \cap_{j =1}^J (t_{j,n} + \lambda_{j,n} I'_j)$ for sufficiently large $n$. Then the following limits hold for each $J \in {\mathbb Z}^+$.
\[
 \lim_{n \rightarrow \infty} \left\|F\left(\sum_{j=1}^J U_{j,n}\right) - \sum_{j=1}^J F(U_{j,n})\right\|_{Z(J_n)} = 0.
\]
\[
 \limsup_{n \rightarrow \infty} \left\|\sum_{j=1}^J U_{j,n} \right\|_{Y(J_n)} \leq \left(\sum_{j=1}^{J}\|U_j\|_{Y(I'_j)}^{p}\right)^{1/p}.
\]
\end{lemma}
\paragraph{Proof} The first limit can be proved by an induction. The case $J=1$ is trivial. If $J>1$, the combination of the following estimate and the induction hypothesis finishes the job.
\begin{eqnarray*}
&& \limsup_{n \rightarrow \infty} \left\|F\left(\sum_{j=1}^J U_{j,n}\right) - F\left(\sum_{j=1}^{J-1} U_{j,n}\right) - F(U_{J,n})\right\|_{Z(J_n)}\\
&=& \limsup_{n \rightarrow \infty} \left\|\left(U_{J,n} \int_0^1 F'\left(\tau U_{J,n} + \sum_{j=1}^{J-1} U_{j,n}\right) d\tau \right)
- \left(U_{J,n} \int_0^1 F'(\tau U_{J,n})d\tau\right) \right\|_{Z(J_n)}\\
&=& \limsup_{n \rightarrow \infty} \left\| \left(U_{J,n} \sum_{j=1}^{J-1} U_{j,n}\right) \left( \int_0^1 \int_0^1
F''\left(\tau U_{J,n} + \tilde{\tau}\sum_{j=1}^{J-1} U_{j,n}\right) d\tilde{\tau} d\tau\right)\right\|_{Z(J_n)}\\
&\leq& \limsup_{n \rightarrow \infty} C_p \left(\sum_{j=1}^{J-1} \|U_{J,n} U_{j,n}\|_{L^\frac{p}{1+s_p} L^\frac{p}{2-s_p}(J_n \times \Rm^3)}\right)
\left( \sum_{j=1}^{J}\|U_{j,n}\|_{Y(J_n)}\right)^{p-2}\\
&\leq& \limsup_{n \rightarrow \infty} C_p \left(\sum_{j=1}^{J-1} \|U_{J,n} U_{j,n}\|_{L^\frac{p}{1+s_p} L^\frac{p}{2-s_p}(J_n \times \Rm^3)}\right)
\left( \sum_{j=1}^{J}\|U_j\|_{Y(I'_j)}\right)^{p-2}\\
&=& 0.
\end{eqnarray*}
In the last step we apply Lemma \ref{general almost orthogonality}. The second limit is a corollary:
\begin{eqnarray*}
 \limsup_{n\rightarrow \infty} \left\|\sum_{j=1}^J U_{j,n}\right\|_{Y(J_n)}^p
 &=& \limsup_{n \rightarrow \infty} \left\|F\left(\sum_{j=1}^J U_{j,n}\right)\right\|_{Z(J_n)}\\
 &\leq& \limsup_{n \rightarrow \infty} \left(\sum_{j=1}^{J}\|F(U_{j,n})\|_{Z(J_n)}\right)\\
 &\leq& \limsup_{n \rightarrow \infty} \left(\sum_{j=1}^{J}\|U_{j,n}\|_{Y(J_n)}^p\right)\\
 &\leq& \sum_{j=1}^{J}\|U_{j}\|_{Y(I'_j)}^p.
\end{eqnarray*}
\begin{remark} \label{arbitrary subsequence}
 The same result still holds if we arbitrarily select a few profiles from $U_j$'s. More precisely, if the inequality $\|U_{j_k}\|_{Y(I'_{j_k})} < \infty$ holds for positive integers $j_1 < j_2 < \cdots < j_m$, then we have
\[
 \lim_{n \rightarrow \infty} \left\|F\left(\sum_{k=1}^m U_{j_k,n}\right) - \sum_{k=1}^m F(U_{j_k,n})\right\|_{Z(J_n)} = 0;
\]
\[
 \limsup_{n \rightarrow \infty} \left\|\sum_{k=1}^m U_{j_k,n} \right\|_{Y(J_n)} \leq \left(\sum_{k=1}^{m}\|U_{j_k}\|_{Y(I'_{j_k})}^{p}\right)^{1/p};
\]
as long as $\displaystyle J_n \subseteq \cap_{k=1}^m (t_{j_k,n} + \lambda_{j_k,n} I'_{j_k})$ holds for all sufficiently large $n$.
\end{remark}

\begin{lemma}[Commutator Estimate] \label{commutator}
Fix $j \in {\mathbb Z}^+$. If $I'_j \subseteq I_j$ so that $\|U_j\|_{I'_j} < \infty$, then the error term
 \[
   e_{j,n} = (\partial_t^2 - \Delta) U_{j,n} - \phi F(U_{j,n})
 \]
satisfies $\displaystyle \lim_{n \rightarrow \infty} \|e_{j,n}\|_{Z(\lambda_{j,n} I'_j + t_{j,n})} = 0$.
\end{lemma}
\begin{proof}
First of all, applying a space-time dilation we have
\[
 \partial_t^2 U_j  -\Delta U_j = \phi_j (x) F(U_j) \; \Longrightarrow (\partial_t^2 - \Delta) U_{j,n} = \phi_j \left(\frac{x}{\lambda_{j,n}}\right) F(U_{j,n}).
\]
Here $\phi_j (x)$ is chosen as in Definition \ref{def of nonlinear profiles}. Therefore we have
\begin{align*}
 \|e_{j,n}\|_{Z(\lambda_{j,n}I'_j + t_{j,n})} = &\left\|\left(\phi_j \left(\frac{x}{\lambda_{j,n}}\right) -\phi(x)\right) F(U_{n,j})\right\|_{Z(\lambda_{j,n}I'_j + t_{j,n})}\\
 = & \left\|\left(\phi_j \left(x\right) -\phi(\lambda_{j,n} x)\right) F(U_j)\right\|_{Z(I'_j)} \rightarrow 0
\end{align*}
by the dominated convergence theorem and the (almost everywhere) point-wise limit $\phi(\lambda_{j,n} x) \rightarrow \phi_j (x)$.
\end{proof}

\section{Compactness Procedure} 
In this section we prove the existence of a critical element and its compactness properties.

\begin{theorem} \label{existence of critical element}
If $Sc(A)$ breaks down at $M$, i.e. the statement $Sc(A)$ holds for all $A\leq M$ but fails for all $A>M$, then there exists a radial solution $u(x,t)$ to (CP1), called a critical element, such that it satisfies
\begin{itemize}
 \item[(i)] Its maximal lifespan is $\Rm$;
 \item[(ii)] It blows up in both time directions with $\|u\|_{Y ([0, \infty))} = \|u\|_{Y ((-\infty,0])}= +\infty$.
 \item[(iii)] The upper bound of its critical Sobolev norm is equal to $M$.
  \[
    \sup_{t \in \Rm} \|(u(\cdot, t), \partial_t u(\cdot, t))\|_{\dot{H}^{s_p} \times \dot{H}^{s_p-1} (\Rm^3)} = M.
  \]
 \item[(iv)] The set $\{(u(\cdot, t), \partial_t u(\cdot, t))| t\in \Rm\}$ is pre-compact in the space $\dot{H}^{s_p} \times \dot{H}^{s_p-1} (\Rm^3)$.
 \end{itemize}
\end{theorem}

\subsection{Setup} \label{sec: setup}
Now let us assume that the statement $Sc(A)$ breaks down at $M$. First of all, we can take a sequence of non-scattering radial solutions $v_n (x,t)$ with maximal lifespans $(-\tilde{T}_n^-, \tilde{T}_n^+)$ so that
\begin{align*}
 &\|v_n\|_{Y([0,\tilde{T}_n^+))} = + \infty;& &\sup_{t \in [0,\tilde{T}_n^+)} \left\|(v_n(\cdot, t), \partial_t v_n (\cdot, t))\right\|_{\dot{H}^{s_p} \times \dot{H}^{s_p-1}} < M + 2^{-n}.&
\end{align*}
The first condition above enable us to find a time $\tilde{t}_n \in [0,\tilde{T}_n^+)$ for each $n$, such that $\|v_n\|_{Y([0,\tilde{t}_n])} = 2^n$. Time translations then give a sequence of new solutions $u_n (x,t) \doteq v_n (x,t+\tilde{t}_n)$. These solutions $\{u_n\}$ satisfy:
\begin{itemize}
\item[(i)] Each solution $u_n$ blows up in the positive time direction, i.e $\|u_n\|_{Y([0,T_n^+))} = + \infty$.
\item[(ii)] $\|u_n\|_{Y ((-T_n^-, 0])} > 2^n$.
\item[(iii)] The inequality $\|(u_n(\cdot, t), \partial_t u_n (\cdot, t)\|_{\dot{H}^{s_p} \times \dot{H}^{s_p-1}} < M + 2^{-n}$ holds for each $t \in [0, T_n^+)$ and for each $t<0$ that satisfies $\|u_n\|_{Y([t,0])} \leq 2^n$.
\end{itemize}
Here the notation $(-T_n^-, T_n^+)$ represents the maximal lifespan of $u_n$. We apply the profile decomposition (Theorem \ref{profile decomposition}) on the sequence of initial data $\{(u_{0,n}, u_{1,n})\} = \{(u_n (\cdot,0), \partial_t u_n (\cdot, 0)\}$, introduce the nonlinear profiles $U_j$ and then define the approximation solutions $U_{j,n}$ as described in Section \ref{sec: profile decomposition3}. The conclusion (iv) of the profile decomposition gives
\begin{equation} \label{norm of v}
 \sum_{j=1}^\infty \|V_j\|_{\dot{H}^{s_p} \times \dot{H}^{s_p-1}}^2 =  \sum_{j=1}^\infty \|(v_{j,0}, v_{j,1})\|_{\dot{H}^{s_p} \times \dot{H}^{s_p-1}}^2 \leq M^2.
\end{equation}
This implies that $\|(U_j (\cdot, t_j), \partial_t U_j (\cdot, t_j))\|_{\dot{H}^{s_p} \times \dot{H}^{s_p-1}} \rightarrow 0$ as $j \rightarrow \infty$ since the definition of the nonlinear profiles implies (if $t_j = \pm \infty$, the norm of $U_j$ is in the sense of limit as $t \rightarrow t_j$)
\[
 \|(U_j (\cdot, t_j), \partial_t U_j (\cdot, t_j))\|_{\dot{H}^{s_p} \times \dot{H}^{s_p-1}} = \|V_j\|_{\dot{H}^{s_p} \times \dot{H}^{s_p-1}}.
\]
According to Theorem \ref{Scattering with small data}, it follows that $U_j$ scatters in both time directions when $j > J_0$ is sufficiently large. In addition, we have
\begin{equation} \label{smaller Y}
 \|U_j\|_{Y(\Rm)} \lesssim_p \|(U_j (\cdot, t_j), \partial_t U_j (\cdot, t_j))\|_H = \|V_j\|_H, \; \hbox{if}\; j > J_0 \; \Longrightarrow \sum_{j=J_0+1}^\infty \|U_j\|_{Y(\Rm)}^p < \infty.
\end{equation}

\subsection{A Single Profile May Survive}
In this subsection we show all but one profile must be zero. If there were at least two nonzero profiles, say $U_{1}$ and $U_2$, we would have
\begin{equation} \label{definition of eps0}
 \eps_0 = \min\left\{\|V_1\|_H, \|V_2\|_H \right\} = \min\left\{\|(v_{1,0}, v_{1,1})\|_H, \|(v_{2,0}, v_{2,1})\|_H \right\} > 0.
\end{equation}
According to \eqref{norm of v}, we can always assume
\begin{equation}
 \sum_{j=J_0 +1}^\infty \|V_j\|_H^2= \sum_{j=J_0+1}^\infty \|(v_{j,0}, v_{j,1})\|_H^2 < \frac{\eps_0^2}{9},
\end{equation}
by possibly increasing the value of $J_0$.

\begin{lemma} \label{upper bound on tail}
 Given any $J > J_0$, we have
 \[
  \limsup_{n \rightarrow \infty} \left\|\sum_{j=J_0+1}^J \left(U_{j,n} (\cdot, 0), \partial_t U_{j,n} (\cdot,0)\right)\right\|_H < \frac{\eps_0}{3}.
 \]
\end{lemma}
\begin{proof} By Remark \ref{nonlinear and linear profiles are close} and Lemma \ref{almost orthogonality 4}, we have
\begin{align*}
 \limsup_{n \rightarrow \infty} \left\|\sum_{j=J_0+1}^J \left(U_{j,n} (\cdot, 0), \partial_t U_{j,n} (\cdot,0)\right)\right\|_H^2 &= \limsup_{n \rightarrow \infty} \left\|\sum_{j=J_0+1}^J \mathbf{S}_L (-t_{j,n}) \Di_{\lambda_{j,n}} (v_{0,j}, v_{1,j}) \right\|_H^2\\
 & = \limsup_{n \rightarrow \infty} \sum_{j=J_0+1}^J \left\| \mathbf{S}_L (-t_{j,n}) \Di_{\lambda_{j,n}} (v_{0,j}, v_{1,j}) \right\|_H^2\\
 & = \limsup_{n \rightarrow \infty} \sum_{j=J_0+1}^J \left\| (v_{0,j}, v_{1,j}) \right\|_H^2 < \frac{\eps_0^2}{9}.
\end{align*}
\end{proof}
\begin{remark} \label{norm of partial initial data}
A similar argument as above shows that if $j_1 < j_2 < \cdots < j_m$ are positive integers, then we have
\[
 \lim_{n \rightarrow \infty} \left\|\sum_{k=1}^m \left(U_{j_k, n} (\cdot, 0), \partial_t U_{j_k, n} (\cdot, 0) \right)\right\|_H^2 = \sum_{k =1}^m \|(v_{0,j_k}, v_{1,j_k})\|_H^2.
\]
\end{remark}
\paragraph{Asymptotic behaviour} If $j>J_0$, then we have already known that the nonlinear profile $U_j$ scatters. We always choose $I'_j = \Rm$ in this case. Otherwise, if $j \leq J_0$, let us consider the behaviour of $U_j(x,t)$ as $t$ goes to $+\infty$. There are two cases:
\begin{itemize}
\item [(I)] $U_j$ scatters in the positive time direction. In this case we choose a time interval
\[
 I'_j = [t_j^-,+\infty) = \left\{\begin{array}{ll} [t_j^-, +\infty), & \hbox{if}\; -t_{j,n}/\lambda_{j,n} \rightarrow t_j \in \Rm, \; \hbox{here we fix $t_j^- \in (-\infty, t_j)\cap I_j$};\cr
 (-\infty, +\infty), & \hbox{if}\; -t_{j,n}/\lambda_{j,n} \rightarrow -\infty;\cr
 [t_j^-, +\infty), & \hbox{if}\; -t_{j,n}/\lambda_{j,n} \rightarrow +\infty,\; \hbox{here we fix $t_j^- \in I_j$}.
 \end{array}\right.
\]
\item[(II)] $U_j$ does not scatter with a maximal lifespan $(-T_j^-, T_j^+)$, thus $t_j < +\infty$. By our assumption on $M$ (if $\lambda_j =1$, see Theorem \ref{existence of critical element}) or Proposition \ref{theorem with constant phi} (if $\lambda_j \in \{0, +\infty\}$), we always have
\[
 \sup_{t\in (t_j, T_j^+)} \|(U_j(\cdot,t), \partial_t U_j(\cdot,t))\|_H \geq M.
\]
As a result, we can find a time $T_j \in (t_j, T_j^+)$ so that
\begin{equation} \label{bad point 1}
 \|(U_j(\cdot, T_j), \partial_t U_j(\cdot,T_j))\|_H > \sqrt{M^2 - \frac{1}{2}[\eta (\eps_0/2)]^2},
\end{equation}
where the function $\eta$ is the one given in Corollary \ref{small data preserve}, and choose
\[
 I'_j = [t_j^-, T_j] =\left\{\begin{array}{ll} [t_j^-, T_j], & \hbox{if}\; -t_{j,n}/\lambda_{j,n} \rightarrow t_j\in \Rm,\; \hbox{here we fix $t_j^- \in (-\infty, t_j)\cap I_j$}; \cr
 (-\infty,T_j], & \hbox{if}\; -t_{j,n}/\lambda_{j,n} \rightarrow -\infty.
 \end{array}\right.
\]
\end{itemize}
In summary, we always have $\|U_j\|_{Y(I'_j)} < \infty$. If $n$ is sufficiently large, we have $-t_{j,n}/\lambda_{j,n}$ is contained in the interior of $I'_j $ for all $j$. Without loss of generality, we can assume that this happens for all $j,n$.

\paragraph{Approximation Solutions} Now let us define
\[
 \bar{t}_n = \sup \left\{t >0 \left| t \in \cap_{j=1}^{J_0} \left(t_{j,n}+\lambda_{j,n} I'_j\right) \right.\right\}.
\]
This is either a positive number or undefined. The second case may happen only if all profiles $U_j$ scatter in the positive time direction. In this case we define $\bar{t}_n \equiv \infty$.  The definition actually implies
\[
 [0, \bar{t}_n]  \subseteq \cap_{j=1}^{\infty} \left(t_{j,n}+\lambda_{j,n} I'_j\right).
\]
According to the profile decomposition and Remark \ref{nonlinear and linear profiles are close}, we can write
\begin{equation} \label{initial value difference}
 (u_{0,n}, u_{1,n}) = \sum_{j=1}^J \left(U_{j,n}(\cdot,0), \partial_t U_{j,n}(\cdot,0)\right) + (w_{0,n}^J,w_{1,n}^J).
\end{equation}
with
\begin{align} \label{limit as J large}
 &\limsup_{n \rightarrow \infty} \left\|(w_{0,n}^J,w_{1,n}^J)\right\|_H \leq M;& &\limsup_{n \rightarrow  \infty} \left\|{\mathbf S}_{L,0} (t)(w_{0,n}^J,w_{1,n}^J)\right\|_{Y(\Rm)} \rightarrow 0 \;\hbox{as}\; J \rightarrow 0.&
\end{align}
Please note that this new error term $(w_{0,n}^J,w_{1,n}^J)$ is different from the one given in the linear profile decomposition. It also covers the error created by the substitution of the linear profiles with their nonlinear counterparts. In addition, the sum $S_{J,n} \doteq \sum_{j=1}^J U_{j,n}$ is a solution of the equation
\begin{equation} \label{equation for sum U}
 \partial_t^2 u - \Delta u = \phi F(u) + Err_{J,n}
\end{equation}
in the time interval $[0,\bar{t}_n]$. Here the error term $Err_{J,n}$ is defined by
\[
 Err_{J,n} = - \phi F\left(\sum_{j=1}^J U_{j,n}\right) + \sum_{j=1}^J \phi F(U_{j,n}) + \sum_{j=1}^J  \left[\partial_t^2 U_{j,n} - \Delta U_{j,n} - \phi F(U_{j,n})\right].
\]
By Lemma \ref{uniform estimate in J}, Lemma \ref{commutator} and the inequality \eqref{smaller Y}, we have
\begin{align}
 \lim_{n \rightarrow \infty} \|Err_{J,n}\|_{Z([0,\bar{t}_n])} & = 0; \label{control of error 1}\\
 \limsup_{n \rightarrow \infty} \left\|S_{J,n}\right\|_{Y([0,\bar{t}_n])}^p & \leq \sum_{j=1}^{J}\|U_j\|_{Y(I'_j)}^{p} \leq \sum_{j=1}^{J_0} \|U_j\|_{Y(I'_j)}^{p} + \sum_{j=J_0+1}^\infty  \|U_j\|_{Y(\Rm)}^{p} < \infty \label{uniform Y bound}.
\end{align}
The upper bound in the second line above is independent of $J$.
\begin{proposition} \label{existence of non-scattering profiles}
The exists at least one profile $U_j$ so that it does not scatter at the positive direction.
\end{proposition}
\begin{proof}
 If it were false, we would have $\bar{t}_n = \infty$ for all $n \in {\mathbb Z}^+$. We can choose a sequence $\{(J_k, n_k)\}_{k \in {\mathbb Z}^+}$, such that
\begin{align*}
 \lim_{k \rightarrow \infty} \left\|{\mathbf S}_{L,0} (t) (w_{0,n_k}^{J_k},w_{1,n_k}^{J_k})\right\|_{Y(\Rm)}  &= 0;\\
 \lim_{k \rightarrow \infty} \|Err_{J_k,n_k}\|_{Z([0,\infty))} & = 0;\\
 \left\|S_{J_k, n_k} \right\|_{Y([0,\infty))}^p &\leq \sum_{j=1}^{J_0} \|U_j\|_{Y(I'_j)}^{p} + \sum_{j=J_0+1}^\infty  \|U_j\|_{Y(\Rm)}^{p} + 1 < \infty.
\end{align*}
 Using the equation \eqref{initial value difference} and \eqref{equation for sum U}, we can apply the long-time perturbation theory on the approximation solutions $S_{J_k, n_k}$, the initial data $(u_{0,n_k}, u_{1,n_k})$ as well as the time interval $[0,\infty)$ and finally conclude that $u_{n_k}$ scatters in the positive time direction if $k$ is sufficiently large. This is a contradiction.
\end{proof}

Now we know $\bar{t}_n \in (0,\infty)$. In addition, for each large $n$, there is a $j \leq J_0$ such that $U_j$ does not scatter in the positive time direction with $\bar{t}_n = \lambda_{j,n} T_j + t_{j,n}$. Passing to a subsequence if necessary, we can assume that the same $j=j_0$ works for all sufficiently large $n$.

\begin{proposition} \label{almost orthogonality20}
The pairs $(U_{j_0,n}(\cdot, \bar{t}_n), \partial_t U_{j_0,n}(\cdot, \bar{t}_n))$ and $(U_{j,n}(\cdot, \bar{t}_n), \partial_t U_{j,n}(\cdot, \bar{t}_n))$ are almost orthogonal in the space $H = \dot{H}^{s_p} \times \dot{H}^{s_p-1}$ if $j \neq j_0$. Namely, we have
\begin{equation*}
\lim_{n \rightarrow \infty} \langle \left(U_{j_0,n}(\cdot, \bar{t}_n), \partial_t U_{j_0,n}(\cdot, \bar{t}_n)\right), \left(U_{j,n}(\cdot, \bar{t}_n), \partial_t U_{j,n}(\cdot, \bar{t}_n)\right)\rangle_H = 0.
\end{equation*}
\end{proposition}
\begin{proof} We have
\begin{align*}
  (U_{j_0,n}(\bar{t}_n), \partial_t U_{j_0,n}(\bar{t}_n)) & = \left(\frac{1}{\lambda_{j_0,n}^{3/2 - s_p}} U_{j_0}\left(\frac{x}{\lambda_{j_0,n}}, T_{j_0}\right), \frac{1}{\lambda_{j,n}^{5/2 -s_p}} \partial_t U_{j_0}\left(\frac{x}{\lambda_{j_0,n}}, T_{j_0}\right)\right);\\
  (U_{j,n}(\bar{t}_n), \partial_t U_{j,n}(\bar{t}_n)) & = \left(\frac{1}{\lambda_{j,n}^{3/2 - s_p}} U_j\left(\frac{x}{\lambda_{j,n}}, \frac{\bar{t}_n -t_{j,n}}{\lambda_{j,n}} \right),\frac{1}{\lambda_{j,n}^{5/2 -s_p}} \partial_t U_{j}\left(\frac{x}{\lambda_{j,n}}, \frac{\bar{t}_n -t_{j,n}}{\lambda_{j,n}}\right)\right).
\end{align*}
Since the dot product is dilation-invariant, we can rewrite the inner product in question into
\[
 \left\langle \left(U_{j_0}(x, T_{j_0}), \partial_t U_{j_0}(x,T_{j_0})\right), \left(\left(\frac{\lambda_{j_0,n}}{\lambda_{j,n}}\right)^{\frac{3}{2} - s_p}\!\!\!\! U_j\left(\frac{\lambda_{j_0,n} x}{\lambda_{j,n}}, t'_n \right), \left(\frac{\lambda_{j_0,n}}{\lambda_{j,n}}\right)^{\frac{5}{2} - s_p} \!\!\!\! \partial_t U_j\left(\frac{\lambda_{j_0,n} x}{\lambda_{j,n}}, t'_n \right) \right) \right\rangle.
\]
Here $\displaystyle t'_n = \frac{\bar{t}_n -t_{j,n}}{\lambda_{j,n}} = \frac{T_{j_0} \lambda_{j_0,n} + t_{j_0,n} - t_{j,n}}{\lambda_{j,n}} \in I'_j$. By the inequality
\[
 |t'_n | \geq - \frac{\lambda_{j_0,n}}{\lambda_{j,n}} |T_{j_0}| +  \frac{| t_{j_0,n} - t_{j,n}|}{\lambda_{j,n}}
\]
and the almost orthogonal condition
\[
 \lim_{n \rightarrow \infty} \left(\frac{\lambda_{j_0,n}}{\lambda_{j,n}} + \frac{\lambda_{j,n}}{\lambda_{j_0,n}} + \frac{| t_{j_0,n} - t_{j,n}|}{\lambda_{j,n}}\right) = + \infty,
\]
we have
\[
 \lim_{n \rightarrow \infty} \left(\frac{\lambda_{j_0,n}}{\lambda_{j,n}} + \frac{\lambda_{j,n}}{\lambda_{j_0,n}} + |t'_n|\right) = + \infty.
\]
Combining this limit with the facts
\begin{itemize}
  \item Each $t'_n$ is contained in the closed interval $I'_j$;
  \item $(U_j (\cdot,t), \partial_t U_j (\cdot, t)) \in C(I'_j; \dot{H}^{s_p} \times \dot{H}^{s_p-1})$;
  \item $U_j (t)$ always scatters in the corresponding time direction whenever $I'_j$ contains a neighbourhood of $\infty$ or $-\infty$;
\end{itemize}
we obtain that the second factor in pairing above converges weakly to zero in $\dot{H}^{s_p} \times \dot{H}^{s_p-1}$. This finishes the proof.
\end{proof}

\paragraph{Approximation Solutions without $U_{j_0,n}$} Now let us define ($J \geq J_0$)
\[
 S'_{J,n} = \sum_{1\leq j\leq J, j\neq j_0} U_{j,n}.
\]
The function $S'_{J,n}$ is the solution to
\begin{equation} \label{approximation equation 2}
 \partial_t^2 u - \Delta u = \phi F(u) + Err'_{J,n}
\end{equation}
The error term is given by
\[
 Err'_{J,n} = - \phi F\left(S'_{J,n}\right) + \sum_{1\leq j\leq J, j\neq j_0} \phi F(U_{j,n}) +  \sum_{1\leq j\leq J, j\neq j_0}  \left[\partial_t^2 U_{j,n} - \Delta U_{j,n} - \phi F(U_{j,n})\right].
\]
By Remark \ref{arbitrary subsequence},  and Lemma \ref{commutator}, we have
\begin{align}
 \lim_{n \rightarrow \infty} \|Err'_{J,n}\|_{Z([0,\bar{t}_n])} & = 0; \label{control of error 2}\\
 \limsup_{n \rightarrow \infty} \left\|S'_{J,n} \right\|_{Y([0,\bar{t}_n])}^p & \leq \sum_{1\leq j\leq J, j \neq j_0}\|U_j\|_{Y(I'_j)}^{p} \leq \sum_{j=1}^{J_0} \|U_j\|_{Y(I'_j)}^{p} + \sum_{j=J_0+1}^\infty  \|U_j\|_{Y(\Rm)}^{p} < \infty \label{uniform Y bound 2}.
\end{align}

\paragraph{Choice of $n(J)$} For each $J > J_0$, we can choose a large positive integer $n(J)$ so that (See \eqref{limit as J large}, \eqref{control of error 1}, \eqref{uniform Y bound}, \eqref{control of error 2}, \eqref{uniform Y bound 2}, Lemma \ref{upper bound on tail} and Lemma \ref{almost orthogonality20})
\begin{align}
& n(J) > J; \\
&\left\|S_{J, n(J)}\right\|_{Y([0,\bar{t}_{n(J)}])} \leq \left(\sum_{j=1}^{J_0} \|U_j\|_{Y(I'_j)}^{p} + \sum_{j=J_0+1}^\infty  \|U_j\|_{Y(\Rm)}^{p} +1\right)^{1/p}; \label{uniform Y bound 31} \\
&\|S'_{J,n(J)} \|_{Y([0,\bar{t}_{n(J)}])} \leq \left(\sum_{j=1}^{J_0} \|U_j\|_{Y(I'_j)}^{p} + \sum_{j=J_0+1}^\infty  \|U_j\|_{Y(\Rm)}^{p} +1\right)^{1/p}; \label{uniform Y bound 3}\\
&\left\|\sum_{j=J_0+1}^J \left(U_{j,n(J)} (\cdot, 0), \partial_t U_{j,n(J)} (\cdot,0)\right)\right\|_H \leq \frac{\eps_0}{3}; \label{upper bound on tail 2} \\
&\|Err_{J,n(J)}\|_{Z([0,\bar{t}_{n(J)}])} \leq 2^{-J}; \label{control of error 31}\\
&\|Err'_{J,n(J)}\|_{Z([0,\bar{t}_{n(J)}])}\leq 2^{-J}; \label{control of error 3}\\
&\left|\left\langle \begin{pmatrix} U_{j_0,n(J)}(\cdot, \bar{t}_{n(J)})\\ \partial_t U_{j_0,n(J)}(\cdot, \bar{t}_{n(J)})\end{pmatrix}, \begin{pmatrix} U_{j,n(J)}(\cdot, \bar{t}_{n(J)})\\ \partial_t U_{j,n(J)}(\cdot, \bar{t}_{n(J)})\end{pmatrix}\right\rangle_H\right| \leq \frac{2^{-J}}{J},\; \hbox{if}\; 1\leq j\leq J,\; j \neq j_0;\label{almost orthogonality 5} \\
&\lim_{J \rightarrow \infty} \left\|{\mathbf S}_{L,0} (t)(w_{0,n(J)}^J, w_{1,n(J)}^J)\right\|_{Y(\Rm)} = 0; \label{limit of Y norm}\\
&\left\|(w_{0,n(J)}^J, w_{1,n(J)}^J)\right\|_H \leq M + 1. \label{uniform bound of w}
\end{align}
Combining the equation \eqref{initial value difference}, \eqref{equation for sum U} and the inequalities \eqref{uniform Y bound 31}, \eqref{control of error 31}, \eqref{limit of Y norm}, we can apply the long-time perturbation theory on the approximation solution $S_{J,n(J)}$, the initial data $(u_{0, n(J)},u_{1,n(J)})$ as well as the time interval $[0,\bar{t}_{n(J)}]$, conclude that $\bar{t}_{n(J)}$ is in the maximal lifespan of $u_{n(J)}$ and
\[
 \lim_{J \rightarrow \infty} \left\|u_{n(J)} - S_{J, n(J)} \right\|_{Y ([0,\bar{t}_{n(J)}])} = 0.
\]
\[
 \lim_{J \rightarrow \infty} \left\|\begin{pmatrix} u_{n(J)} (\cdot, \bar{t}_{n(J)})\\ \partial_t u_{n(J)} (\cdot, \bar{t}_{n(J)})\end{pmatrix} -
 \begin{pmatrix} S_{J,n(J)} (\cdot, \bar{t}_{n(J)})\\ \partial_t S_{J,n(J)}(\cdot, \bar{t}_{n(J)})\end{pmatrix}- {\mathbf S}_L (\bar{t}_{n(J)}) \begin{pmatrix} w_{0,n(J)}^J\\ w_{1,n(J)}^J\end{pmatrix} \right\|_H = 0,
\]
if $J$ is sufficiently large. Therefore we have
\begin{align}
 & \limsup_{J \rightarrow \infty} \left\|\left(u_{n(J)} (\cdot, \bar{t}_{n(J)}), \partial_t u_{n(J)} (\cdot, \bar{t}_{n(J)})\right)\right\|_H\nonumber\\
  &\qquad\qquad =  \limsup_{J \rightarrow \infty} \left\| \begin{pmatrix} S_{J,n(J)} (\cdot, \bar{t}_{n(J)})\\ \partial_t S_{J,n(J)}(\cdot, \bar{t}_{n(J)})\end{pmatrix} + {\mathbf S}_L (\bar{t}_{n(J)}) \begin{pmatrix} w_{0,n(J)}^J\\ w_{1,n(J)}^J\end{pmatrix}\right\|_H.\label{identity 101}
\end{align}
By \eqref{limit of Y norm}, \eqref{uniform bound of w}, Lemma \ref{almost orthogonal of w} and the identity
\[
 \begin{pmatrix}  U_{j_0,n(J)}(\cdot, \bar{t}_{n(J)})\\  \partial_t U_{j_0, n(J)}(\cdot, \bar{t}_{n(J)})\end{pmatrix} = \Di_{\lambda_{j_0,n(J)}} \begin{pmatrix}  U_{j_0} \left(\cdot, T_{j_0}\right)\\  \partial_t U_{j_0}\left(\cdot, T_{j_0}\right)\end{pmatrix},
\]
We have
\[
 \lim_{J \rightarrow \infty} \left\langle  \begin{pmatrix}  U_{j_0,n(J)}(\cdot, \bar{t}_{n(J)})\\  \partial_t U_{j_0, n(J)}(\cdot, \bar{t}_{n(J)})\end{pmatrix},{\mathbf S}_L (\bar{t}_{n(J)}) \begin{pmatrix} w_{0,n(J)}^J\\ w_{1,n(J)}^J\end{pmatrix} \right\rangle_H  = 0.
\]
Combining this with \eqref{almost orthogonality 5} and \eqref{identity 101}, we obtain
\begin{align}
 &\limsup_{J \rightarrow \infty} \left\|\left(u_{n(J)} (\cdot, \bar{t}_{n(J)}), \partial_t u_{n(J)} (\cdot, \bar{t}_{n(J)})\right)\right\|_H^2 \nonumber\\
 = &  \limsup_{J \rightarrow \infty} \left(\left\| \begin{pmatrix} U_{j_0,n(J)}(\cdot, \bar{t}_{n(J)})\\  \partial_t U_{j_0, n(J)}(\cdot, \bar{t}_{n(J)})\end{pmatrix} \right\|_H^2 + \left\| \begin{pmatrix}
S'_{J,n(J)}(\cdot, \bar{t}_{n(J)})\\  \partial_t S'_{J,n(J)}(\cdot, \bar{t}_{n(J)}) \end{pmatrix}+ {\mathbf S}_L (\bar{t}_{n(J)}) \begin{pmatrix} w_{0,n(J)}^J\\ w_{1,n(J)}^J\end{pmatrix} \right\|_H^2 \right)\nonumber\\
 = &  \left\| \begin{pmatrix}  U_{j_0} \left(\cdot, T_{j_0}\right)\\  \partial_t U_{j_0}\left(\cdot, T_{j_0}\right)\end{pmatrix}\right\|_H^2+   \limsup_{J \rightarrow \infty} \left\| \begin{pmatrix}
S'_{J,n(J)}(\cdot, \bar{t}_{n(J)})\\  \partial_t S'_{J,n(J)}(\cdot, \bar{t}_{n(J)}) \end{pmatrix}+ {\mathbf S}_L (\bar{t}_{n(J)}) \begin{pmatrix} w_{0,n(J)}^J\\ w_{1,n(J)}^J\end{pmatrix}\right\|_H^2. \label{lower bound of norm 1}
\end{align}
We still need to find a lower bound on the second term above. First of all, let us consider the initial data
\begin{equation} \label{initial data 2}
 \begin{pmatrix} u'_{0, J}\\ u'_{1,J} \end{pmatrix} \doteq \begin{pmatrix} S'_{J, n(J)} (\cdot, 0)\\ \partial_t S'_{J,n(J)}(\cdot, 0)\end{pmatrix} + \begin{pmatrix} w_{0,n(J)}^J\\ w_{1,n(J)}^J\end{pmatrix}
\end{equation}
and define $u'_J$ to be corresponding solution to (CP1). According to \eqref{upper bound on tail 2} we have
\begin{equation} \label{lower bound on initial data 1}
 \left\|\begin{pmatrix} u'_{0, J}\\ u'_{1,J} \end{pmatrix}\right\|_H  \geq \left\|\begin{pmatrix} S'_{J_0, n(J)} (\cdot, 0)\\ \partial_t S'_{J_0,n(J)}(\cdot, 0)\end{pmatrix} + \begin{pmatrix} w_{0,n(J)}^J\\ w_{1,n(J)}^J\end{pmatrix}\right\|_H - \frac{\eps_0}{3}.
\end{equation}
Observing \eqref{limit of Y norm}, \eqref{uniform bound of w} and the identity $\left(\frac{- t_{j,n(J)}}{\lambda_{j,n(J)}} \in I'_j\right)$
\[
  \begin{pmatrix} S'_{J_0, n(J)} (\cdot, 0)\\ \partial_t S'_{J_0, n(J)}(\cdot, 0) \end{pmatrix} = \sum_{1\leq j\leq J_0, j \neq j_0} \Di_{\lambda_{j,n(J)}}  \begin{pmatrix}  U_{j} \left(\cdot, \frac{- t_{j,n(J)}}{\lambda_{j,n(J)}}\right)\\  \partial_t U_{j}\left(\cdot, \frac{ - t_{j,n(J)}}{\lambda_{j,n(J)}}\right)\end{pmatrix},
\]
we are able to apply Lemma \ref{almost orthogonal of w} and conclude
\[
 \lim_{J \rightarrow \infty} \left\langle \begin{pmatrix} S'_{J_0, n(J)} (\cdot, 0)\\ \partial_t S'_{J_0,n(J)}(\cdot, 0)\end{pmatrix}, \begin{pmatrix} w_{0,n(J)}^J\\ w_{1,n(J)}^J\end{pmatrix} \right\rangle_H = 0.
\]
Next we combine this limit, Remark \ref{norm of partial initial data}, the definition \eqref{definition of eps0} and obtain
\begin{align*}
\liminf_{J \rightarrow \infty} \left\|\begin{pmatrix} S'_{J_0, n(J)} (\cdot, 0)\\ \partial_t S'_{J_0,n(J)}(\cdot, 0)\end{pmatrix} + \begin{pmatrix} w_{0,n(J)}^J\\ w_{1,n(J)}^J\end{pmatrix}\right\|_H^2 & \geq \liminf_{J \rightarrow \infty} \left\|\begin{pmatrix} S'_{J_0, n(J)} (\cdot, 0)\\ \partial_t S'_{J_0,n(J)}(\cdot, 0)\end{pmatrix} \right\|_H^2\\
& = \sum_{1\leq j \leq J_0, j\neq j_0} \|(v_{j,0}, v_{j,1})\|_H^2 \geq \eps_0^2.
\end{align*}
Plugging this lower bound into the inequality \eqref{lower bound on initial data 1}, we have
\[
 \displaystyle \liminf_{J \rightarrow \infty} \|(u'_{0,J}, u'_{1,J})\|_H \geq \frac{2\eps_0}{3}.
\]
According to Corollary \ref{small data preserve}, this implies
\begin{equation} \label{lower bound on bar u}
  \liminf_{J \rightarrow \infty} \left(\inf_{t} \left\|\left(u'_{J} (\cdot, t), \partial_t u'_{J} (\cdot, t)\right)\right\|_H \right) \geq \eta(\eps/2).
\end{equation}
Recalling the fact that $S'_{J,n}$ satisfies the equation \eqref{approximation equation 2} and combining this with the identity \eqref{initial data 2}, the estimates \eqref{uniform Y bound 3}, \eqref{control of error 3}, \eqref{limit of Y norm}, we can apply the long-time perturbation theory on the approximation solution $S'_{J,n(J)}$, the initial data $(u'_{0,J}, u'_{1,J})$ as well as the time interval $[0, \bar{t}_{n(J)}]$, and conclude that $\bar{t}_{n(J)}$ is contained in the maximal lifespan of $u'_J$ for large $J$ with
\begin{equation}
  \lim_{J \rightarrow \infty} \left\|\begin{pmatrix} u'_{J} (\cdot, \bar{t}_{n(J)})\\ \partial_t u'_{J} (\cdot, \bar{t}_{n(J)})\end{pmatrix} -
 \begin{pmatrix} S'_{J,n(J)} (\cdot, \bar{t}_{n(J)})\\ \partial_t S'_{J,n(J)} (\cdot, \bar{t}_{n(J)})\end{pmatrix}  - {\mathbf S}_L (\bar{t}_{n(J)}) \begin{pmatrix} w_{0,n(J)}^J\\ w_{1,n(J)}^J\end{pmatrix} \right\|_H = 0,
\end{equation}
This implies
\begin{align*}
  \limsup_{J \rightarrow \infty} \left\| \begin{pmatrix}
S'_{J,n(J)}(\cdot, \bar{t}_{n(J)})\\  \partial_t S'_{J,n(J)}(\cdot, \bar{t}_{n(J)}) \end{pmatrix}+ {\mathbf S}_L (\bar{t}_{n(J)}) \begin{pmatrix} w_{0,n(J)}^J\\ w_{1,n(J)}^J\end{pmatrix}\right\|_H^2 & =  \limsup_{J \rightarrow \infty} \left\|\begin{pmatrix} u'_{J} (\cdot, \bar{t}_{n(J)})\\ \partial_t u'_{J} (\cdot, \bar{t}_{n(J)})\end{pmatrix}\right\|_H^2\\
& \geq [\eta(\eps_0/2)]^2.
\end{align*}
In the last step above, we use the lower limit \eqref{lower bound on bar u}. Combining this with \eqref{bad point 1} and \eqref{lower bound of norm 1}, we finally obtain
\[
 \limsup_{J \rightarrow \infty} \left\|\left(u_{n(J)} (\cdot, \bar{t}_{n(J)}), \partial_t u_{n(J)} (\cdot, \bar{t}_{n(J)})\right)\right\|_H^2 \geq
  M^2 - \frac{1}{2}[\eta (\eps_0/2)]^2 + [\eta(\eps_0/2)]^2 > M^2.
\]
This contradicts with our assumption (iii) on $u_n$. (Please see Subsection \ref{sec: setup})

\subsection{Extraction of the Critical Element}

At this point there is only one nonzero profile $U_1$ with a maximal lifespan $I_1$. The profile decomposition can be rewritten into
\begin{align}
 (u_{0,n}, u_{1,n}) & = \left(U_{1,n} \left(\cdot, 0\right), \partial_t U_{1,n} \left(\cdot , 0\right)\right) + (w_{0,n}, w_{1,n}) \label{single profile decomposition}\\
 & = \left(\frac{1}{\lambda_{1,n}^{3/2-s_p}} U_{1} \left(\frac{\cdot}{\lambda_{1,n}}, \frac{-t_{1,n}}{\lambda_{1,n}}\right), \frac{1}{\lambda_{1,n}^{5/2-s_p}} \partial_t U_{1} \left(\frac{\cdot}{\lambda_{1,n}}, \frac{-t_{1,n}}{\lambda_{1,n}}\right)\right) + (w_{0,n}, w_{1,n}). \nonumber
\end{align}
Here the nonlinear profile $U_1$ and the error terms $(w_{0,n}, w_{1,n})$ satisfy
\begin{align}
 &\limsup_{n \rightarrow \infty} \|(w_{0,n},w_{1,n})\|_H \leq M; \label{limit of H norm of w}\\
 &\lim_{n \rightarrow \infty} \|{\mathbf S}_{L,0} (t)(w_{0,n}, w_{1,n})\|_{Y(\Rm)} = 0;\label{limit of Ystar zero} \\
 &\lim_{t \rightarrow t_1} \|(U_1(\cdot, t), \partial_t U_1 (\cdot, t))\|_H = \|(v_{0,1}, v_{1,1})\|_H \leq M. \label{upper bound on norm of u1 1}
\end{align}
In addition, the approximation solution
\[
 U_{1,n}(x,t) = \frac{1}{\lambda_{1,n}^{3/2-s_p}} U_{1} \left(\frac{x}{\lambda_{1,n}}, \frac{t-t_{1,n}}{\lambda_{1,n}}\right), \quad (x, t) \in \Rm^3 \times (t_{1,n} + \lambda_{1,n} I_1)
\]
satisfies the equation
\begin{equation} \label{approximation equation 11}
 \partial_t^2 u - \Delta u = \phi F(u) + Err_{1,n},
\end{equation}
so that if $I'_1 \subseteq I_1$ is any time interval with $\|U_1\|_{Y(I'_1)} < \infty$, then we have
\begin{align}
 & \|U_{1,n}\|_{Y(t_{1,n} + \lambda_{1,n} I'_1)}  = \|U_1\|_{Y(I'_1)} < \infty,  \label{upper bound of Y norm 11}\\
 &\lim_{n \rightarrow \infty} \|Err_{1,n}\|_{Z(t_{1,n} + \lambda_{1,n} I'_1)} = 0.\label{control of error 11}
\end{align}
It turns out that the nonlinear profile $U_1$ is exactly the critical element we are looking for. Let us first consider its asymptotic behaviour.

\paragraph{Failure of scattering in both two time directions} We have already know that $U_1$ fails to scatter in the positive time direction by Proposition \ref{existence of non-scattering profiles}. The way in which a nonlinear profile is defined implies that $t_1 = \lim_{n \rightarrow \infty} -t_{1,n}/\lambda_{1,n} < +\infty$. We still need to consider the asymptotic behaviour of $U_1$ in the negative time direction. If $U_1$ scattered in the negative time direction, we could choose an interval $I'_1 = (-\infty, t_1^+] \subset I_1$, so that
\begin{itemize}
 \item $t_1^+ > t_1$, thus $(-\infty,0] \subseteq t_{1,n} + \lambda_{1,n} I'_1$ holds for each sufficiently large $n$;
 \item $\|U_1\|_{Y(I'_1)} < + \infty$.
\end{itemize}
Combining these with the profile decomposition \eqref{single profile decomposition}, the fact that $U_{1,n}$ satisfies the equation \eqref{approximation equation 11}, the inequality \eqref{upper bound of Y norm 11}, and the limits \eqref{limit of Ystar zero}, \eqref{control of error 11},  we are able to apply the long-time perturbation theory on the approximation solution $U_{1,n}$, the initial data $(u_{0,n}, u_{1,n})$ and the time interval $(-\infty,0]$, finally to conclude that $(-\infty,0]$ is contained in the maximal lifespan of $u_n$ if $n$ is large with
\[
 \lim_{n \rightarrow \infty} \|u_n - U_{1,n}\|_{Y((-\infty,0])} = 0.
\]
This means that $\displaystyle \limsup_{n \rightarrow \infty} \|u_n\|_{Y((-\infty,0])} \leq \|U_1\|_{Y(I'_1)} < \infty$. This contradicts our consumption (iii) in Subsection \ref{sec: setup}. One direct corollary is that $t_1$ is finite.

\paragraph{Upper bound on the Sobolev norm} Since $t_1$ is finite, we have $ \|(U_1(\cdot, t_1), \partial_t U_1 (\cdot, t_1))\|_H \leq M$ by \eqref{upper bound on norm of u1 1}. Now we claim
\begin{equation} \label{uniform H bound for U1}
 \sup_{t \in I_1} \|(U_1(\cdot, t), \partial_t U_1 (\cdot, t))\|_H \leq M.
\end{equation}
In fact, this is a direct corollary of the following lemma.
\begin{lemma} \label{orthogonality 01}
Given any time $T \in I_1$, we have
\[
 \left\|\left( U_{1} (\cdot, T), \partial_t U_{1} (\cdot, T)\right)\right\|_H^2 + \limsup_{n \rightarrow \infty} \left\|\left(w_{0,n}, w_{1,n}\right)\right\|_H^2 \leq M^2.
\]
This is equivalent to
\[
  \left(\sup_{t \in I_1}\left\|\left( U_{1} (\cdot, t), \partial_t U_{1} (\cdot, t)\right)\right\|_H\right)^2 + \limsup_{n \rightarrow \infty} \left\|\left(w_{0,n}, w_{1,n}\right)\right\|_H^2 \leq M^2.
\]
\end{lemma}
\begin{proof}
We only need to consider the case $T < t_1$. Since we can deal with the opposite case $T>t_1$ in the same way and finally the case $T=t_1$ by continuity. We started by picking up a time $t_1^+ \in (t_1, \infty) \cap I_1$ and choosing $I'_1 = [T,t_1^+] \subset I_1$. Therefore we have $\|U_1\|_{Y(I'_1)} < \infty$. One can also check that $\bar{t}_n \doteq \lambda_{1,n} T + t_{1,n} < 0$ and $[\bar{t}_n,0] \subseteq t_{1,n} + \lambda_{1,n} I'_1$ hold if $n$ is sufficiently large. As a result, we have
\begin{equation} \label{Y bound of u1n}
 \|U_{1,n}\|_{Y([\bar{t}_n,0])} \leq \|U_1\|_{Y(I'_1)} < \infty.
\end{equation}
Now we are able to apply the long-time perturbation theory on the approximation solution $U_{1,n}$, initial data $(u_{0,n},u_{1,n})$ and the time interval $[\bar{t}_n,0]$ if $n$ is large by the profile decomposition \eqref{single profile decomposition}, the fact that $U_{1,n}$ satisfies the approximation equation \eqref{approximation equation 11},  the limits \eqref{limit of Ystar zero}, \eqref{control of error 11} and the inequality \eqref{Y bound of u1n}. We conclude that $[\bar{t}_n,0]$ is contained in the lifespan of $u_n$ if $n$ is large and
\begin{align}
 &\lim_{n \rightarrow \infty} \left\|u_n - U_{1,n} \right\|_{Y([\bar{t}_n,0])} = 0;&\label{near time zero} \\
 &\lim_{n \rightarrow \infty} \left\|\begin{pmatrix} u_n (\cdot, \bar{t}_n)\\ \partial_t u_n (\cdot, \bar{t}_n)\end{pmatrix} - \begin{pmatrix} U_{1,n} (\cdot, \bar{t}_n)\\ \partial_t U_{1,n} (\cdot, \bar{t}_n)\end{pmatrix} - {\mathbf S}_L (\bar{t}_n)\begin{pmatrix} w_{0,n} \\ w_{1,n}\end{pmatrix} \right\|_H = 0.& \label{H norm estimate 1}
\end{align}
Next we combine the limit \eqref{near time zero} with the inequality \eqref{Y bound of u1n} and obtain
\[
 \limsup_{n \rightarrow \infty} \|u_n\|_{Y([\bar{t}_n,0])} = \limsup_{n \rightarrow \infty} \|U_{1,n}\|_{Y([\bar{t}_n,0])} \leq \|U_1\|_{Y(I'_1)} < \infty.
\]
Therefore by our assumption (iii) in Subsection \ref{sec: setup} we have
\begin{equation}  \label{location of bartn}
  \limsup_{n \rightarrow \infty} \left\| \left(u_n (\cdot, \bar{t}_n), \partial_t u_n (\cdot, \bar{t}_n)\right) \right\|_H \leq M.
\end{equation}
In addition, the limit \eqref{H norm estimate 1} gives
\begin{align*}
 \limsup_{n \rightarrow \infty} \left\|\begin{pmatrix} u_n (\cdot, \bar{t}_n)\\ \partial_t u_n (\cdot, \bar{t}_n)\end{pmatrix} \right\|_H^2 & =
  \limsup_{n \rightarrow \infty} \left\| \begin{pmatrix} U_{1,n} (\cdot, \bar{t}_n)\\ \partial_t U_{1,n} (\cdot, \bar{t}_n)\end{pmatrix} + {\mathbf S}_L (\bar{t}_n)\begin{pmatrix} w_{0,n} \\ w_{1,n}\end{pmatrix} \right\|_H^2\\
  & = \limsup_{n \rightarrow \infty} \left\| \Di_{\lambda_{1,n}} \begin{pmatrix} U_{1} (\cdot, T)\\ \partial_t U_{1} (\cdot, T)\end{pmatrix} + {\mathbf S}_L (\bar{t}_n)\begin{pmatrix} w_{0,n} \\ w_{1,n}\end{pmatrix} \right\|_H^2\\
  & \geq \lim_{n \rightarrow \infty} \left\| \Di_{\lambda_{1,n}} \begin{pmatrix} U_{1} (\cdot, T)\\ \partial_t U_{1} (\cdot, T)\end{pmatrix}\right\|_H^2 +  \limsup_{n \rightarrow \infty} \left\| {\mathbf S}_L (\bar{t}_n)\begin{pmatrix} w_{0,n} \\ w_{1,n}\end{pmatrix} \right\|_H^2\\
  & = \left\|\begin{pmatrix} U_{1} (\cdot, T)\\ \partial_t U_{1} (\cdot, T)\end{pmatrix}\right\|_H^2 + \limsup_{n \rightarrow \infty} \left\|\begin{pmatrix} w_{0,n} \\ w_{1,n}\end{pmatrix} \right\|_H^2.
\end{align*}
Here we use \eqref{limit of H norm of w}, \eqref{limit of Ystar zero} and apply Lemma \ref{almost orthogonal of w}. Finally we plug the upper bound \eqref{location of bartn} into the left hand of the inequality above and finish the proof.
\end{proof}

\paragraph{A solution to (CP1)} According to the definition of a nonlinear profile, $U_1$ is either a solution to (CP1) or a solution to the equation $\partial_t^2 u -\Delta u = c F(u)$, where $c$ is a constant. If it were the latter case, we could apply Proposition \ref{theorem with constant phi} by the upper bound \eqref{uniform H bound for U1} and conclude that $U_1$ scatters in both two time directions. This contradicts the already-known asymptotic behaviour of $U_1$.

\paragraph{The least upper bound on the norm} Now we can conclude
\begin{equation}
 \sup_{t \in I_1} \|(U_1(\cdot, t), \partial_t U_1 (\cdot, t))\|_H = M. \label{sup of UH}
\end{equation}
It has been known in \eqref{uniform H bound for U1} that the least upper bound above does not exceed $M$. Therefore we only need to show this least upper bound can not be smaller than $M$. This is trivial since we have assumed that Sc(M) holds and we have known that $U_1$ fails to scatter in both two time directions. Combining this least upper bound and Lemma \ref{orthogonality 01}, we obtain
\begin{equation} \label{strong limit of w}
 \lim_{n \rightarrow \infty} \|(w_{0,n}, w_{1,n})\|_H = 0.
\end{equation}

\paragraph{Compactness of initial data} Now we can give a compactness result.

\begin{proposition} \label{compactness of initial data}
 Let $\{(u_{0,n}, u_{1,n})\}_{n \in {\mathbb Z}^+}$ be a sequence of radial initial data and $\{u_n\}$ be their corresponding solutions to (CP1), so that $\{u_n\}$ satisfies the conditions (i), (ii) and (iii) listed at the beginning of Subsection \ref{sec: setup}. Then there exists a subsequence of the initial data, so that it converges strongly in the space $\dot{H}^{s_p} \times \dot{H}^{s_p-1}(\Rm^3)$.
\end{proposition}

\begin{proof}
 We have already known a subsequence, still denoted $\{(u_{0,n}, u_{1,n})\}$, so that the single-profile representation \eqref{single profile decomposition} holds. Combining the limits $\lambda_{1,n} \rightarrow \lambda_1 = 1$, $-t_{1,n}/\lambda_{1,n} \rightarrow t_1 \in I_1$, the fact $(U_1(\cdot,t), \partial_t U_1(\cdot,t)) \in C(I_1; \dot{H}^{s_p} \times \dot{H}^{s_p-1})$ and the strong convergence \eqref{strong limit of w}, we have the strong limit
\[
 (u_{0,n}, u_{1,n}) \rightarrow (U_1(\cdot, t_1), \partial_t U_1 (\cdot, t_1))\quad \hbox{in}\; \dot{H}^{s_p} \times \dot{H}^{s_p-1}(\Rm^3).
\]
\end{proof}

\paragraph{Almost periodicity of the critical element} Now we are able to conclude the set
\[
 \{(u(\cdot,t), \partial_t u(\cdot, t)| t\in I_1\}
\]
is pre-compact in the space $\dot{H}^{s_p} \times \dot{H}^{s_p-1}(\Rm^3)$. In fact, if $\{t_n\}_{n \in {\mathbb Z}^+}$ is a sequence of time in $I_1$, then the time-translated solutions $U_1(x, t+t_n)$ solve (CP1) and satisfy the conditions (i), (ii), (iii) listed at the beginning of Subsection \ref{sec: setup}, with initial data $(U_1(\cdot,t_n), \partial_t U_1 (\cdot,t_n))$. Now we can apply Proposition \ref{compactness of initial data}, conclude that the sequence $\{(U_1(\cdot,t_n), \partial_t U_1 (\cdot,t_n))\}$ has a convergent subsequence in the space $\dot{H}^{s_p} \times \dot{H}^{s_p-1}(\Rm^3)$ and thus finish the proof.

\paragraph{Global existence and the completion of proof} According to Remark \ref{lifespan lower bound for compact set}, the compactness result above implies that there exists a positive constant $\eps$, such that
\[
 t_0 \in I_1 \Longrightarrow [t_0 - \eps, t_0 + \eps] \subseteq I_1.
\]
This means that $I_1 = \Rm$. Collecting all information about $U_1$ we have obtained, finally we are able to finish the proof of Theorem \ref{existence of critical element}.

\begin{remark}
 We could prove a stronger version of Theorem \ref{existence of critical element} without the radial assumption in a similar way if we assumed that Theorem \ref{theorem with phi pm 1} still holds without the radial assumption. The latter has not been proved yet, as far as the author knows, although we expect that it is still true.
\end{remark}

\section{Further Properties of the Critical Element}

In this section we show that the critical element has to satisfy further regularity conditions. The argument is similar to the one we used for the special case $\phi(x) \equiv \pm 1$. The radial assumption plays an important role in this argument. If $u(x,t)$ is a radial function, then we use the notation $u(r,t)$ for the value $u(x,t)$ when $|x|=r$. The main idea is that if $u$ is a radial solution to
\[
 \partial_t^2 u - \Delta u = F (|x|,t),
\]
then the function $w (r,t) \doteq r u(r,t)$ is a solution to the one-dimensional wave equation
\begin{equation} \label{eqn for w}
 \partial_t^2 w - \partial_r^2 w = r F(r,t).
\end{equation}
A direct calculation shows
\begin{lemma} [See Lemma 4.2 in \cite{shen2}]\label{eqn between w and u}
Let $(u(x,t_0), \partial_t u(x,t_0))$ be radial so that
\[
 \nabla u(\cdot, t_0), \partial_t u(\cdot, t_0) \in L^2(\{x \in \Rm^3: a < |x| < b\})
\]
for any $0 < a < b < \infty$, then we have the identity
\[
 \frac{1}{4\pi}\int_{a < |x| < b} (|\nabla u|^2 + |\partial_t u|^2) dx = \left(\int_a^b [(\partial_r w)^2 + (\partial_t w)^2]dr\right) +
 \left(a u^2(a) - b u^2(b)\right)
\]
holds. Here we take the value of the functions at time $t_0$.
\end{lemma}
First of all, we claim that $u$ is always more regular than $\dot{H}^{s_p} \times \dot{H}^{s_p-1}$ away from the origin.
\begin{proposition}
Assume $3 < p < 5$. Let $u$ be a radial solution to the wave equation
\[
 \partial_t^2 u - \Delta u = F(|x|,t),
\]
defined for all $t \in \Rm$ so that
\begin{itemize}
 \item The set $\{(u(\cdot, t), \partial_t u(\cdot, t))| t \in \Rm\}$ is pre-compact in the space $\dot{H}^{s_p} \times \dot{H}^{s_p-1} (\Rm^3)$.
 \item The function $F(|x|,t)$ has a finite $Z(I)$ norm for any bounded time interval $I$ and satisfies the inequality $\left|F(|x|,t)\right| \leq C_0 |x|^{-\frac{2p}{p-1}}$ for all $(x,t) \in (\Rm^3 \setminus \{0\})\times \Rm$.
\end{itemize}
If we define $w(r,t) = ru(r,t)$, then $\partial_r w(\cdot,t), \partial_t w(\cdot, t)\in C(\Rm_t; L^2 ([R,\infty)))$ for $R>0$ with
\begin{align}
 & \int_R^{4R} \left[(\partial_r w(r,t))^2 + (\partial_t w(r,t))^2\right]\, dr \nonumber\\
 & \leq  \frac{1}{2} \int_R^{4R} \left[ \left(\int_0^\infty (r+t) F(r +t, t_0 -t) dt\right)^2+ \left(\int_0^\infty (r+t) F(r +t, t_0 +t) dt\right)^2\right]\, dr \label{w estimate1}\\
 & \leq C_1 R^{-2(1-s_p)}. \nonumber
\end{align}
This implies that $(u(\cdot,t), \partial_t u(\cdot,t)) \in C(\Rm; \dot{H}^1 \times L^2 (\Rm^3 \setminus B(0,R)))$ with
\begin{equation}
\int_{R < |x| < 4R} \left(|\nabla u(x,t)|^2+ |\partial_t u(x,t)|^2\right) dx \leq C_2 R^{-2(1-s_p)}.
\end{equation}
Here the constants $C_1$ and $C_2$ are independent of $t$ and $R$.
\end{proposition}

\begin{proof}
 A similar result in the special case $F(|x|,t) = |u|^{p-1} u$ has been proved in the author's previous work \cite{shen2}. This general case can be proved in exactly the same way, please refer to Section 4 of the work mentioned above. The main ingredients of the proof include the transformation $u \rightarrow w$ as given above, the standard method to deal with one-dimensional wave equation via path integrals, Duhamel's formula, strong Huygens' principle and smooth approximation techniques.
\end{proof}

\paragraph{Uniform decay of $u$ as $r \rightarrow \infty$} We recall the explicit formula to solve the initial-value problem of the one-dimensional wave equation \eqref{eqn for w} and obtain that $w = ru$ satisfies
\begin{align}
 w(r, t_0) = &  \frac{1}{2}\left[w\left(\frac{r}{2}, t_0 - \frac{r}{2}\right) + w \left(\frac{3r}{2}, t_0 - \frac{r}{2}\right)\right] + \frac{1}{2} \int_{r/2}^{3r/2} \partial_t w \left(s,t_0-\frac{r}{2}\right)\, ds \nonumber \\
 & + \frac{1}{2} \int_0^{r/2} \int_{\frac{r}{2}+t}^{\frac{3r}{2}-t} s F\left(s, t_0 - \frac{r}{2} +t\right) \, ds \,dt. \label{explicit one d wave}
\end{align}
Let us fix $\beta_0 = \frac{3}{2} - s_p = 2/(p-1)$. For each $\beta \in [\beta_0, 1)$ we define a function
\[
 f_\beta: \Rm^+ \rightarrow [0,\infty)\cup\{\infty\},\qquad f_\beta (r) = \sup_{t \in \Rm, |x|\geq r} |x|^{\beta} |u(x,t)|,
\]
which helps us compare the decay rate of $u$ with that of $|x|^{-\beta}$ as $|x| \rightarrow \infty$. Let us assume that $f_\beta (r)$ is always finite for $r>0$ and that $f_\beta (r) \rightarrow 0$ as $r \rightarrow \infty$, which are true at least for $\beta = \beta_0$, thanks to Lemma \ref{pointwise radial Hs} and Lemma \ref{uniform decay for compact set}. Now we recall $F(|x|,t)= \phi(x) |u(x,t)|^{p-1} u(x,t)$, use the upper bound $|u(x,t)| \leq |x|^{-\beta} f(|x|) \leq |x|^{-\beta} f(r/2)$ and the inequality \eqref{w estimate1} on the right hand of \eqref{explicit one d wave}, divide both sides by $r^{1-\beta}$ and finally obtain
\begin{equation}
r^{\beta} |u(r,t_0)| \leq \left[g(\beta) + C_p f_\beta^{p-1} \left(\frac{r}{2}\right) r^{2-(p-1)\beta} \right] f_\beta \left(\frac{r}{2}\right). \label{r beta u}
\end{equation}
Here
\[
g(\beta) = \frac{1}{2}\left[\left(\frac{3}{2}\right)^{1-\beta} + \left(\frac{1}{2}\right)^{1-\beta} \right] < 1.
\]
We observe that the right hand right of \eqref{r beta u} is a non-increasing function of $r$, take the least upper bound on both sides for all $r > r_0$ and obtain
\[
 f_\beta (r_0)\leq  \left[g(\beta) + C_p f_\beta^{p-1} \left(\frac{r_0}{2}\right) r_0^{2-(p-1)\beta} \right] f_\beta \left(\frac{r_0}{2}\right).
\]
Since $2-(p-1)\beta \leq 0$ and $\displaystyle \lim_{r_0 \rightarrow \infty} f_\beta (r_0/2) = 0$, we have
\[
 f_\beta (r) \leq \frac{g(\beta) + 1}{2} f_\beta (r/2)
\]
when $r$ is sufficiently large. This implies that $f_\beta$ decays at a rate at least comparable to that of a small negative power of $r$. As a result, we can increase the value of $\beta$, iterate this argument, and eventually push $\beta$ to $1$. The final conclusion is summarized in the following proposition. Please see Section 7 of \cite{shen2} for more details on this argument.

\begin{proposition} \label{enery channel 2}
 Assume $3 < p < 5$. Let $u$ be a radial solution to the wave equation
\[
 \partial_t^2 u - \Delta u = F(|x|,u,t),
\]
defined on all $t \in \Rm$ so that
\begin{itemize}
 \item The set $\{(u(\cdot, t), \partial_t u(\cdot, t))| t \in \Rm\}$ is pre-compact in the space $\dot{H}^{s_p} \times \dot{H}^{s_p-1} (\Rm^3)$.
 \item The function $F(|x|,u,t)$ satisfies $\left|F(|x|,u,t)\right| \leq |u|^p$.
\end{itemize}
Then we have three constants $A\in \Rm$, $C_3, C_4 > 0$ independent of $t$, $r$ and $x$, such that
\begin{itemize}
 \item The solution satisfies
 \begin{align*}
  &|u(x,t)| \leq \frac{C_3}{|x|},& &\left|u(x,t) - \frac{A}{|x|}\right| \lesssim \frac{1}{|x|^{p-2}};&
 \end{align*}
 \item We have an estimate on the local energy
 \[
  \int_{r <|x|<4r} \left(|\nabla u(x,t)|^2 + |\partial_t u(x,t)|^2 \right) \,dx \leq C_4 r^{-1}.
 \]
\end{itemize}
\end{proposition}

\section{Ground State}

In this section we construct a ground state with a similar asymptotic behaviour to the critical element $u$ when $|x| \rightarrow \infty$.

\subsection{Existence of the Ground State} \label{sec: existence}

\begin{proposition} \label{existence of the ground state}
Fix $p>3$. Given any constant $A$, there exists a radius $R = R(\phi, A) \geq 0$ and a radial solution $W_A \in C^2 (\Rm^3 \setminus \bar{B}(0,R))$ to the elliptic equation $- \Delta W = \phi (x) |W|^{p-1} W$, such that
\begin{itemize}
\item The behaviour of $W_A (x)$ as $|x| \rightarrow \infty$ is characterized by
\begin{align*}
 &\left|W_A (x) - \frac{A}{|x|}\right| \lesssim \frac{1}{|x|^{p-2}},& &|\nabla W_A (x)| \lesssim \frac{1}{|x|^2}.&
\end{align*}
\item If $R(\phi, A) > 0$, then $\displaystyle \limsup_{|x| \rightarrow R(\phi, A)^+} \left|W_A (x)\right| = +\infty$.
\end{itemize}
\end{proposition}

\begin{proof}
 Let us rewrite $W_A (x)$ into the form $\displaystyle W_A (x) = \frac{A + \rho (|x|)}{|x|}$. Here $\rho$ is a function defined for large positive real numbers with limits $\rho(r), \rho'(r) \rightarrow 0$ as $r \rightarrow \infty$.  The elliptic equation $- \Delta W = \phi (x) |W|^{p-1} W$ can be rewritten in term of $\rho$:
 \begin{equation}
  \rho''(r) = - \frac{\phi(r) F(\rho(r) +A)}{r^{p-1}}, \qquad F(u) = |u|^{p-1}u. \label{eqn of rho}
 \end{equation}
 The first step is to show this equation has a unique solution defined on the interval $[R_1,\infty)$ via a fixed-point argument, where $R_1$ is a large radius to be determined later. We define a complete metric space
 \[
  X = \left\{\rho: \rho \in C([R_1,\infty); [-1,1]),\; \lim_{r \rightarrow +\infty} \rho(r) = 0 \right\}
 \]
with the distance $\displaystyle d(\rho_1, \rho_2) = \sup_{r \in [R_1,\infty)} \left|\rho_1(r) - \rho_2 (r)\right|$ and a map
\[
 \left({\mathbf L} \rho\right)(r) = \int_{r}^\infty \int_s^\infty \left(- \frac{\phi(t) F(\rho(t) +A)}{t^{p-1}}\right) \,dt\, ds.
\]
Since the absolute value of the integrand never exceeds $(1+|A|)^p/t^{p-1}$, this integral defines a continuous function on $[R_1,\infty)$. In addition, we have
\begin{align}
 \left|\left({\mathbf L} \rho\right)(r)\right| & \leq \int_{r}^\infty \int_s^\infty \frac{(1+|A|)^p}{t^{p-1}} \,dt\, ds \leq \frac{C_p (1+|A|)^p}{r^{p-3}};\label{estimate of rho} \\
 \left|\left({\mathbf L} \rho_1\right)(r) - \left({\mathbf L} \rho_2\right)(r)\right| & \leq  \int_{r}^\infty \int_s^\infty \frac{p(1+|A|)^{p-1}d(\rho_1, \rho_2)}{t^{p-1}} \,dt\, ds \leq \frac{C_p (1+|A|)^{p-1}}{r^{p-3}}d(\rho_1, \rho_2). \nonumber
\end{align}
As a result, if we choose a sufficiently large $R_1 = R_1 (A, p)$, then the map $\mathbf L$ is a contraction map on the space $X$. This enables us to apply a fixed-point argument and find a solution $\rho$ to the equation \eqref{eqn of rho} defined on $[R_1,\infty)$. The formula $W_A (x) = (A+\rho(|x|))/|x|$ then gives a $C^2$ solution to the elliptic equation $- \Delta W = \phi (x) |W|^{p-1} W$ defined for $|x|>R_1$. Furthermore, we have an estimate on $\rho'(r)$ when $r\geq R_1$:
\begin{equation} \label{estimate of rho prime}
 \left|\rho'(r)\right| = \left|\int_r^\infty \left(\frac{\phi(t) F(\rho(t) +A)}{t^{p-1}}\right) \,dt\right| \leq \frac{C_p (1+|A|)^p}{r^{p-2}}.
\end{equation}
Combing this upper bound on $\rho'(r)$ with the upper bound \eqref{estimate of rho} on $\rho(r)$, we obtain the asymptotic behaviour of $W_A (x)$ as $|x|$ is large:
\begin{align*}
 \left|W_A(x) - \frac{A}{|x|}\right| & = \frac{\left|\rho(|x|)\right|}{|x|} \lesssim_{p,A} \frac{1}{|x|^{p-2}};\\
 \left|\nabla W_A(x)\right| & = \left|\frac{\rho'(|x|)}{|x|} - \frac{A + \rho(|x|)}{|x|^2}\right|\lesssim_{p,A} \frac{1}{|x|^2}.
\end{align*}
The second step is to extend the solution $\rho(r)$ to its maximal interval of existence $(R(\phi,A), \infty)$. We still need to prove
\[
 \limsup_{|x| \rightarrow R(\phi, A)^+} \left|W_A (x)\right| = +\infty
\]
if $R(\phi, A) > 0$. This is equivalent to saying the upper limit of $|\rho(r)|$ as $r \rightarrow R(\phi, A)^+$ is infinity. If this were false, then we can assume $|\rho(r)| \leq M$ for all $r > R(\phi, A)$. But this implies that both $\rho''(r)$ and $\rho'(r)$ are also bounded when $r$ approaches the blow-up point $R(\phi, A)$ from the right, according to the equation \eqref{eqn of rho}. This contradicts with a basic theory of ordinary differential equations.
\end{proof}

\subsection{Classification of the Ground States} \label{sec: classification}

The ground states $W$ obtained via Proposition \ref{existence of the ground state} can be classified into three categories
\begin{itemize}
 \item[(I)] The ground state $W$ is only defined for points a certain distance away from the origin, i.e. $R(\phi, A) > 0$. Proposition \ref{existence of the ground state} guarantees that $W(x)$ is unbounded when $|x| \rightarrow R(\phi, A)^+$. Therefore $W$ is not in the space $\dot{H}^{s_p} (\Rm^3)$, thanks to Lemma \ref{pointwise radial Hs}. More precisely, it is impossible to find a radial function $u \in \dot{H}^{s_p}(\Rm^3)$, such that $u(x) = W(x)$ for all $x$ with $|x|> R(\phi, A)$.
 \item[(II)] The ground state $W$ is well-defined everywhere except for at the origin . But we have
 \[
  \limsup_{|x| \rightarrow 0^+} |x|^{\frac{3}{2} -s_p}|W(x)| > 0.
 \]
According to Lemma \ref{uniform decay for compact set}, we know $W \notin \dot{H}^{s_p} (\Rm^3)$. Examples of this type are given by the special case $\phi(x) \equiv 1$. Please see Section 9 of \cite{shen2} for more details.
\item[(III)] The ground state $W$ is well-defined on $\Rm^3 \setminus \{0\}$, and satisfies
\[
  \lim_{|x| \rightarrow 0^+} |x|^{\frac{3}{2} -s_p}|W(x)| = 0.
 \]
It turns out that this ground state $W$ must be $C^2$ in the whole space $\Rm^3$, as shown in the Proposition \ref{C2 smooth at the origin} below. A Combination of this $C^2$ smoothness with the decay rate of the gradient $\nabla W$ near infinity guarantees that $W \in \dot{W}^{1,q}$ for all $q > 3/2$. By the Sobolev embedding we have $W \in \dot{H}^s (\Rm^3)$ for all $s \in (1/2,1]$, in particular for $s = s_p$. One example in this category can be given explicitly by the function
\[
 W(x) = \frac{3^{2/(p-1)}}{\sqrt{3|x|^2 + 1}},
\]
which solves the elliptic equation $-\Delta W = 3^{-2(5-p)/(p-1)} W^5$. As a result, it also solves the elliptic equation $-\Delta W = \phi_p (x) |W|^{p-1} W$ if we choose
\[
 \phi_p (x) = 3^{-2(5-p)/(p-1)}|W(x)|^{5-p} = (3|x|^2 + 1)^{(p-5)/2}.
\]
\end{itemize}

\begin{remark} \label{cate 3}
It turns out that if Proposition \ref{existence of the ground state} does not give any nontrivial $\dot{H}^{s_p}(\Rm^3)$ ground state, i.e. the only ground state $W_A(x)$ constructed above that falls in category (III) is the trivial one $W_0(x) \equiv 0$, then the elliptic equation $-\Delta W = \phi(x) |W|^{p-1} W$ will not admit any nontrivial radial solution in the space $\dot{H}^{s_p} (\Rm^3)$ at all. This will be proved in Section \ref{sec: final}.
\end{remark}

\begin{remark}
The existence of nontrivial radial solutions to the elliptic equation $-\Delta u = \phi(x) |u|^{p-1} u$ and their asymptotic behaviour as $|x| \rightarrow \infty$ has been considered in some previous works such as \cite{NY, nscon}. In particular, T. Kusano and M. Naito claim that if $\phi: [0,\infty) \rightarrow R^+$ satisfies
\begin{itemize}
\item $\phi(r) \in C[0,\infty) \cap C^1 (0,\infty)$;
\item $\frac{d}{dr}\left[r^{(5-p)/2} \phi(r)\right]$ is nonnegative for all $t>0$ but not identically zero;
\end{itemize}
then every radial $C^2$ solution to the elliptic equation  $-\Delta u = \phi(x) |u|^{p-1} u$ is oscillatory, i.e. it has a zero in any neighbourhood of infinity. It is clear that none of these solutions can share the same asymptotic behaviour as $A/|x|$, as long as $A \neq 0$. Therefore for these $\phi$'s Proposition \ref{existence of the ground state} does not give any nontrivial $\dot{H}^{s_p}(\Rm^3)$ ground state. In other words, the elliptic equation $-\Delta W = \phi(x) |W|^{p-1} W$ does not admit any nontrivial radial solution in the space $\dot{H}^{s_p} (\Rm^3)$. Please see their work \cite{KN1} and citation therein for more details.
\end{remark}

\begin{lemma} \label{almost local boundedness}
Let $W \in C^2 (B(0,r_0)\setminus \{0\})$ be a radial solution to the elliptic equation $-\Delta W = q(|x|) W$, where $q(r)$ is a continuous function defined on $(0,r_0)$ satisfying $\displaystyle \lim_{r \rightarrow 0^+} r^2 q(r) = 0$. In addition, we assume that there is a constant $\eps \in (0,1/2)$, such that $\lim_{|x| \rightarrow 0^+} |x|^{1-\eps} W(x) = 0$. Then there exist two constants $r_1 \in (0,r_0)$ and $C_1 > 0$, such that the inequality $|W(x)| \leq C_1 |x|^{-\eps}$ holds for all $0 < |x| < r_1$.
\end{lemma}

\begin{proof}
This is trivial if $W \equiv 0$, thus we assume $W$ is not identically zero in any neighbourhood of the origin. First of all, we define a new function $v: (0,r_0) \rightarrow \Rm$ by $v(|x|) = |x|^{1-\eps} W(x)$. According to the assumption on $W$ we have $v(r) \rightarrow 0$ as $r \rightarrow 0^+$. A basic calculation shows that $v$ satisfies the equation
\begin{equation}
  v''(r) + \frac{2\eps}{r} v'(r) =\eta (r) v(r). \label{equation of v}
\end{equation}
Here $\displaystyle \eta (r) = \frac{\eps (1-\eps)}{r^2} - q(r)$. By the assumption on $q(r)$, there exists a small positive number $r_1 \in (0,r_0)$, such that $\eta (r) > 0$ for all $r \in (0,r_1)$.

\paragraph{Step 1} We claim that $v(r)$ has neither a positive local maximum nor a negative local minimum on $(0,r_1)$. If $v(r)$ had a positive local maximum at $r = r_2 \in (0,r_1)$, then we would have $v''(r_2) \leq 0$, $v'(r_2) = 0$ and $\eta(r_2) v(r_2) > 0$. This violates the equation \eqref{equation of v}. The same argument rules out the existence of any negative local minimum.

\paragraph{Step 2} The function $v(r)$ is either always positive or always negative in the interval $(0,r_1)$. If this were false, then by continuity we would find a number $r_2 \in (0,r_1)$, so that $v(r_2) = 0$. Since $v(r)$ is a nontrivial $C^2$ function defined on $(0,r_2]$ with $\displaystyle \lim_{r \rightarrow 0^+} v(r) = 0$ and $v(r_2) = 0$, it must have either a local positive maximum or a local negative minimum in the interval $(0,r_2)$. This is a contradiction. Without loss of generality, we assume $v(r) > 0$ for all $r \in (0,r_1)$.

\paragraph{Step 3} The derivative $v'(r) > 0$ for all $r \in (0, r_1)$. In fact, a negative derivative $v'(r_2) < 0$ at a point $r_2 \in (0,r_1)$ would imply the existence of a positive local maximum in the interval $(0,r_2)$, since we have $v(r) > 0$ for $r \in (0,r_1)$ and $v(r) \rightarrow 0$ as $r \rightarrow 0$. Therefore we have $v'(r) \geq 0$ for all $r \in (0,r_1)$. Furthermore, if $v'(r_2) = 0$ at a point $r_2 \in (0,r_1)$, then this point would be a local minimum of $v'(r)$ thus $v''(r_2) = 0$. Again this contradicts with the equation \eqref{equation of v}, since we have assumed $v(r_2) > 0$.

\paragraph{Step 4} Now in the interval $(0,r_1)$ we can rewrite \eqref{equation of v} into an inequality
\[
 v''(r) + \frac{2\eps}{r} v'(r) > 0 \quad \Longrightarrow \quad \frac{d}{dr}\left\{\ln [v'(r)]\right\} > \frac{-2\eps}{r}.
\]
Integrating this from $r$ to $r_1$, we obtain
\[
 \ln[v'(r_1)] - \ln[v'(r)] > -2\eps \ln r_1 + 2\eps \ln r\quad \Longrightarrow \quad v'(r) < C r^{-2\eps}, \quad \hbox{if}\; 0 < r < r_1.
\]
Here $C = r_1^{2\eps} v'(r_1)$. Combining this inequality with $\displaystyle \lim_{r\rightarrow 0^+} v(r) = 0$, we obtain that if $0 < r = |x|< r_1$, then
\begin{equation}
 0 < v(r) < C_1 r^{1-2\eps} \quad \Longrightarrow \quad | W(x)| < C_1 |x|^{-\eps}.  \label{upper bound of y}
\end{equation}
Here $C_1 = C/(1-2\eps)$.
\end{proof}

\begin{proposition} \label{C2 smooth at the origin}
Fix $p \in (3,5]$. Let $W \in C^2 (\Rm^3 \setminus \{0\})$ be a radial solution to the elliptic equation $-\Delta W = \phi(x) |W|^{p-1} W$ so that
\[
 \lim_{|x| \rightarrow 0^+} |x|^{\frac{3}{2} -s_p}|W(x)| = 0.
\]
Then we can extend the domain of $W$ to the whole space $\Rm^3$ by continuity so that $W \in C^2 (\Rm^3)$ gives a classic solution to the elliptic equation above.
\end{proposition}
\begin{proof}
Let us choose a small constant $0 < \eps < \min\{1/p, 1-\frac{2}{p-1}\}$. The conditions on $W$ enable us to apply Lemma \ref{almost local boundedness} with $q(|x|)  = \phi(x) |W(x)|^{p-1}$ and to obtain an estimate $|W(x)| \leq C_1 |x|^{-\eps}$ for small $|x| < r_1$. For simplicity we define $y: \Rm^+ \rightarrow \Rm$ by $y(|x|) = W(x)$. A basic calculation shows that $y$ satisfies the equation
\begin{equation}
 (ry)'' + r \phi(r) |y|^{p-1} y = 0. \label{second form of y}
\end{equation}
In addition, the upper bound on $W$ gives $|y(r)| \leq C_1 r^{-\eps}$. We combine this estimate with our assumption $p \eps < 1$ and obtain that $|(ry)''| = r|\phi(r)| |y|^p \leq C_1^p$ is bounded for all $0 < r < \min\{r_1,1\}$. As a result, the limit
\[
\lim_{r \rightarrow 0^+} (ry)' = C_2
\]
exists with $\left|(ry)' - C_2\right| \leq C_1^p r$ for small $r$. A basic integration then shows
\[
 \left|ry(r)-C_2 r\right| \leq \frac{C_1^p}{2} r^2 \quad \Longrightarrow \quad \left|y(r)- C_2\right| \leq \frac{C_1^p}{2} r.
\]
Therefore the function $W$ extends to a continuous function on $\Rm^3$.  Since the right hand of $-\Delta W = \phi(x) |W|^{p-1} W$ is continuous, we can gain two derivatives and conclude $W \in C^2 (\Rm^3)$ by basic knowledge in Laplace's equation.
\end{proof}

\section{Non-existence of a Critical Element}

Given any critical element $u$, Proposition \ref{enery channel 2} guarantees the existence of a real number $A$, so that we have
\[
 \left|u(x,t) - \frac{A}{|x|}\right| \lesssim \frac{1}{|x|^{p-2}}.
\]
when $x$ is large. According to Proposition \ref{existence of the ground state}, we also have a solution $W_A (x)$ to the elliptic equation $-\Delta W = \phi (x) |W|^{p-1} W$ with the same asymptotic behaviour when $|x| \rightarrow \infty$:
\[
 \left|W_A(x) - \frac{A}{|x|}\right| \lesssim \frac{1}{|x|^{p-2}}.
\]
Therefore the critical element $u(x,t)$ and the ground state $W_A(x)$ are close to each other when $|x|$ is large:
\[
 \left|u(x,t) - W_A (x) \right| \lesssim \frac{1}{|x|^{p-2}}.
\]
Our goal is to show $u (x,t) \equiv W_A (x)$. The argument consists of two steps. In the first step we show the identity holds for very large $x$. Then in the second step we prove that the identity has to hold for all $x, t$ wherever $W_A(x)$ is still defined. This immediately gives a contradiction unless $W_A(x) \in \dot{H}^{s_p} (\Rm^3)$, thus finishes the proof of our main theorem. Each step mentioned above is summarized into a theorem, which works for a more general nonlinear term as well. Both theorems are given in Subsection \ref{sec: abstract th} and can be proved by the ``channel of energy'' method in exactly the same way as in the special case $\phi (x) \equiv 1$. Thus we will give the statements only and omit the details of proof. Please see \cite{secret} and Section 8 of \cite{shen2} for more details.

\subsection{Abstract Theorems} \label{sec: abstract th}

\paragraph{Assumptions} Assume $3 < p < 5$. Let $W \in C^2 (\{x \in \Rm^3: |x| > R_W\})$ be a radial solution to the elliptic equation
\[
 -\Delta W = F(|x|, W),
\]
where $F: [0,\infty) \times \Rm \rightarrow \Rm$ is a continuous function satisfying
\begin{align*}
 \left| F(r, u) \right| & \leq |u|^p;\\
 \left| F(r,u_1) - F(r,u_2) \right| & \leq C_5 \left|u_1 - u_2\right|\left(|u_1|^{p-1} + |u_2|^{p-1}\right);
\end{align*}
so that the inequalities $\displaystyle |W(x)| \lesssim \frac{1}{|x|}$,  $\displaystyle |\nabla W(x)| \lesssim \frac{1}{|x|^2}$ hold when $|x|$ is large. We say $u(x,t)$ is a solution to the equation $\partial_t^2 u -\Delta u = F(|x|,u)$ in the time interval $I$, if $(u(\cdot,t), \partial_t u(\cdot,t)) \in C(I; \dot{H}^{s_p} \times \dot{H}^{s_p-1} (\Rm^3))$, with a finite norm $\|u\|_{Y(J)} < \infty$ for any closed bounded interval $J \subseteq I$, so that the integral equation
\[
 u(\cdot, t) = {\mathbf S}_{L,0} (t)(u(\cdot,t_0),\partial_t u(\cdot,t_0)) + \int_{t_0}^t \frac{\sin [(t-\tau)\sqrt{-\Delta}]}{\sqrt{-\Delta}} F(\cdot, u(\cdot, \tau)) d\tau
\]
holds for all $t, t_0 \in I$.

\begin{theorem} \label{identity near infinity}
Let $W(x)$ and $F(r,u)$ be as above . Suppose $u(x,t)$ is a radial solution to the equation $\partial_t^2 u - \Delta u = F(|x|, u)$ defined for all $t \in \Rm$ so that
\begin{itemize}
  \item[(I)] 
  The following inequality holds for each $t\in \Rm$ and $r>0$.
  \begin{equation}
   \int_{r <|x|<4r} \left(|\nabla u(x,t)|^2 + |\partial_t u(x,t)|^2 \right) dx \leq C_4 r^{-1}.
  \end{equation}
  \item[(II)] The functions $u(x,t)$ and $W(x)$ are very close to each other as $|x|$ is large.
  \begin{equation}
   |u (x,t) - W (x)| \lesssim \frac{1}{|x|^{p-2}}.
  \end{equation}
\end{itemize}
Then there exists a constant $R_0 > R_W$ such that \footnote{In fact, the constant $R_0$ is determined solely by $p, C_4, C_5, W$.}
\begin{align*}
 &u(x, t) - W(x) = 0,& & u_1(x, t)= 0&
\end{align*}
hold for all $t\in \Rm$ and $|x| > R_0$.
\end{theorem}

\paragraph{Essential Radius of Support} If the pairs $(u(x,t), u_1(x,t))$ and $(W(x),0)$ coincide for large $x$, we can define the essential radius of support for their difference by
\[
 R(t) = \min \{R \geq R_W: (u(x,t) - W (x), \partial_t u(x,t))=(0,0)\, \hbox {holds for}\, |x|> R\}.
\]

\begin{theorem} \label{behaviour of compact solution}
Let $W(x)$, $F(r,u)$ be as above and $I$ be a time interval containing a neighbourhood of $t_0$. Suppose $u(x,t)$ is a radial solution of the equation $\partial_t^2 u - \Delta u = F(|x|, u)$
on a time interval $I$ satisfying
\begin{itemize}
  \item [(I)] $(u(x,t), \partial_t u(x,t)) \in C(I; \dot{H}^1 \times L^2 (\Rm^3 \setminus B(0,R_W)))$.
  \item [(II)] The pair $(u(x,t_0) - W (x), \partial_t u(x,t_0))$ is compactly supported with an essential radius of support $R(t_0) > R_1 > R_W$.
\end{itemize}
Then there exists a positive constant $\tau$, which is solely determined by $p, R_1, C_5$ and $W$, such that the identity
\[
 R(t) = R(t_0) + |t-t_0|
\]
holds either for each $t \in [t_0, t_0+ \tau]\cap I$ or for each $t \in [t_0-\tau,t_0]\cap I$.
\end{theorem}

\subsection{A Critical Element must be a Ground State} 

In this subsection we apply the theorems above and finally finish the proof of our main theorem.

\paragraph{Step 1} Given a critical element $u(x,t)$, we have already obtained a ground state $W_A(x)$ defined for $|x|>R(\phi, A)$ so that $|u(x,t)-W_A(x)| \lesssim 1/|x|^{p-2}$. Now we apply Theorem \ref{identity near infinity} on $u$ and $W_A (x)$. Our conclusion is that there is a radius $R_0 > R(\phi, A)$, so that
\[
 \left(u(x, t), \partial_t u (x, t)\right) = \left(W_A (x), 0\right), \quad \hbox{for}\; |x| > R_0,\; t\in \Rm.
\]
As a result, we know the essential radius of support $R(t)$ for the difference $(u(x,t) - W_A (x), \partial_t u(x,t))$ is well-defined and satisfies the inequality$R(t) \leq R_0$ for all $t$.

\paragraph{Step 2} Next we prove $u(x,t) = W(x)$ for all time $t$ and $|x| > R_W$. If this were false, we would find a time, say $t=0$, so that $R(t) > R_W$ and deduce a contradiction. We start by choosing $R_1 = [R(0) + R_W]/2$ and applying Theorem \ref{behaviour of compact solution} with $I = \Rm$ and $t_0 =0$. Without loss of generality, we assume that the radius of support $R(t)$ increases linearly in the positive time direction for a time period $\tau$. More precisely we have
\[
 R(t) = R(0) + t, \quad \hbox{for}\; t \in [0,\tau].
\]
Since $R(\tau) > R(0) > R_1$, we are able to apply Theorem \ref{behaviour of compact solution} at time $t_0 = \tau$ again with the same constant $R_1$. Our conclusion is that $R(t)$ has to increase in a linear manner in at least one time direction for the same time period $\tau$ as $t$ moves away from $t=\tau$. This must be the positive time direction since we have known that $R(t)$ decreases as $t \in [0,\tau]$ decreases. Therefore we obtain
\[
 R(t) = R(0) + t, \quad \hbox{for}\; t \in [0,2\tau].
\]
Repeating this argument, we have $R(t) = R(0) + t$ for all $t > 0$. This contradicts the already-known uniform upper bound $R(t) \leq R_0$.

\subsection{Completion of the proof} \label{sec: final} 

Finally we finish the proof of our main theorem. In fact, the following statements are all equivalent to each other
\begin{itemize}
 \item[(I)] The elliptic equation $-\Delta W = \phi(x) |W|^{p-1} W$ admit a nonzero radial solution $W$ in the space $\dot{H}^{s_p}(\Rm^3) \cap C^2 (\Rm^3)$.
 \item[(II)] Proposition \ref{existence of the ground state} gives at least one nontrivial $\dot{H}^{s_p}(\Rm^3)$ ground state; i.e. There exists $A \neq 0$, so that the solution $W_A(x)$ constructed
  in Section \ref{sec: existence} is contained in the space $\dot{H}^{s_p} (\Rm^3)$.
 \item[(III)] There exists a critical element, as described in Theorem \ref{existence of critical element}.
 \item[(IV)] There exists a radial solution $u$ to (CP1) with a maximal lifespan $(-T_-, T_+)$, so that it fails to scatter in the positive time direction but satisfies
  \[
   \sup_{t \in [0,T_+)} \left\|(u(\cdot,t),\partial_t u(\cdot,t))\right\|_{\dot{H}^{s_p} \times \dot{H}^{s_p-1}(\Rm^3)} < + \infty.
  \]
\end{itemize}
It suffices to prove that 
\[
 (I) \Longrightarrow (IV) \Longrightarrow (III) \Longrightarrow (II) \Longrightarrow (I).
\]
The first step here is trivial. Because a solution as described in (I) immediately gives an example that verifies (IV). The second step $(IV) \Longrightarrow (III)$ is exactly the compactness procedure we carried on in the first half of this paper. The final step $(II) \Longrightarrow (I)$ has been done in Section \ref{sec: classification} when we discussed the classification of ground states. Finally let us show $(III) \Longrightarrow (II)$. 
Given a critical element $u$, we can find a real number $A$ and a solution $W_A(x) \in C^2(\{x: |x|>R(\phi, A)\})$ to the elliptic equation $-\Delta W = \phi(x) |W|^{p-1} W$ by the argument in the earlier part of this section, so that
\begin{itemize}
 \item[(a)] We have $u(x,t) = W_A (x)$ for all $|x|>R(\phi, A),\; t \in \Rm$.
 \item[(b)] If $R(\phi,A)>0$, then we also have $\displaystyle \limsup_{|x| \rightarrow R(\phi,A)^+} |W_A(x)| = + \infty$.
\end{itemize}
If $R(\phi, A)$ were positive, then we could combine (a) and (b) above and obtain 
\[
 \limsup_{|x| \rightarrow R(\phi,A)^+} |u(x,0)| = + \infty. 
\]
This contradicts Lemma \ref{pointwise radial Hs}. Therefore we must have $R(\phi, A) = 0$ and $W_A = u(\cdot,0) \in \dot{H}^{s_p} (\Rm)$.

\end{document}